\documentclass[12pt, reqno]{amsart}

\usepackage[utf8]{inputenc}
\usepackage[T2A]{fontenc}
\usepackage[ukrainian,russian,english]{babel}

\usepackage[margin=2.9cm]{geometry}
\usepackage{enumitem}
	
\usepackage{amssymb}
\usepackage{amscd}
\usepackage{graphicx}
\usepackage{xcolor}
\usepackage[all]{xy}

	\newcommand{\theoremname}{Theorem}%
	\newcommand{\sublemmaname}{Lemma}%
	\newcommand{\lemmaname}{Lemma}%
	\newcommand{\definitionname}{Definition}%
	\newcommand{\examplename}{Example}%
	\newcommand{\problemname}{Problem}%
	\newcommand{\conjecturename}{Conjecture}%
	\newcommand{\assumptionname}{Assumption}%
	\newcommand{\corollaryname}{Corollary}%
	\newcommand{\propositionname}{Proposition}%
	\newcommand{\remarkname}{Remark}%

\newtheorem{theorem}[subsection]{\protect\theoremname}
\newtheorem{lemma}[subsection]{\protect\lemmaname}
\newtheorem{sublemma}[subsubsection]{\protect\sublemmaname}
\newtheorem{proposition}[subsection]{\protect\propositionname}

\newtheorem{definition}[subsection]{\protect\definitionname}
\newtheorem{example}[subsection]{\protect\examplename}

\newtheorem{remark}[subsection]{\protect\remarkname}

\makeatletter
\@addtoreset{equation}{section}
\@addtoreset{figure}{section}
\@addtoreset{table}{section}
\makeatother

\makeatletter
\newcommand\testshape{family=\f@family; series=\f@series; shape=\f@shape.}
\def\myemphInternal#1{\if n\f@shape%
\begingroup\itshape #1\endgroup\/%
\else\begingroup\bfseries #1\endgroup%
\fi}
\def\myemph{\futurelet\testchar\MaybeOptArgmyemph}
\def\MaybeOptArgmyemph{\ifx[\testchar \let\next\OptArgmyemph
                 \else \let\next\NoOptArgmyemph \fi \next}
\def\OptArgmyemph[#1]#2{\index{#1}\myemphInternal{#2}}
\def\NoOptArgmyemph#1{\myemphInternal{#1}}
\makeatother

\newcommand{\bR}{\mathbb{R}}
\newcommand{\bN}{\mathbb{N}}

\newcommand{\Int}{\mathrm{Int}}

\newcommand\Xsp{X}
\newcommand\Ysp{Y}

\newcommand\Usp{U}
\newcommand\Vsp{V}

\newcommand\chartMap{\varphi} 

\newcommand{\Btrans}{B}

\newcommand\dif{h}
\newcommand\kdif{k}
\newcommand{\Partition}{\Delta}

\newcommand{\leaf}{\omega}
\newcommand{\strip}{S}

\newcommand{\stripSurf}{Z}
\newcommand{\preStripSurf}{\stripSurf_0}
\newcommand{\bdX}{X}
\newcommand{\bdY}{Y}
\newcommand{\aind}{{\alpha}}
\newcommand{\Aind}{\mathbf{A}}

\newcommand{\stInd}{{\lambda}}
\newcommand{\StInd}{\Lambda}

\newcommand{\bdGlueInd}{{\gamma}}
\newcommand{\BdGlueInd}{\Gamma}

\newcommand{\id}{\mathrm{id}}

\newcommand{\qmap}{q}

\renewcommand{\emptyset}{\varnothing}

\newcommand{\leveld}[1]{u_{#1}}

\newcommand\HS{\mathcal{H}(\Partition)}
\newcommand\HZS{\mathcal{H}_{\id}(\Partition)}

\newcommand\Homeo{\mathcal{H}}

\newcommand\Gr{G}
\newcommand\AutG{\mathrm{Aut}(\Gr)}
 
\newcommand\ori{\mathbf{or}}
\newcommand\lori{l}
\newcommand\tori{\tau}

\newcommand\edge[2]{\{#1,#2\}}
\newcommand\eiso{\varepsilon}
\newcommand\viso{\nu}
\newcommand\mZZ{\{\pm1\}}

\newcommand\hcl[1]{\mathrm{c}(#1)}
\newcommand\JL{J_{\leaf}}

\begin{document}

\author{Sergiy Maksymenko, Eugene Polulyakh, Yuliya Soroka}

\title{Homeotopy groups of one-dimensional foliations on surfaces}

\begin{abstract}
Let $Z$ be a non-compact two-dimensional manifold obtained from a family of open strips $\mathbb{R}\times(0,1)$ with boundary intervals by gluing those strips along their boundary intervals.
Every such strip has a foliation into parallel lines $\mathbb{R}\times t$, $t\in(0,1)$, and boundary intervals, whence we get a foliation $\Delta$ on all of $Z$.
Many types of foliations on surfaces with leaves homeomorphic to the real line have such ``striped'' structure.
That fact was discovered by W.~Kaplan (1940-41) for foliations on the plane $\mathbb{R}^2$ by level-set of pseudo-harmonic functions $\mathbb{R}^2 \to \mathbb{R}$ without singularities.

Previously, the first two authors studied the homotopy type of the group $\mathcal{H}(\Delta)$ of homeomorphisms of $Z$ sending leaves of $\Delta$ onto leaves, and shown that except for two cases the identity path component $\mathcal{H}_{0}(\Delta)$ of $\mathcal{H}(\Delta)$ is contractible.
The aim of the present paper is to show that the quotient $\mathcal{H}(\Delta)/ \mathcal{H}_{0}(\Delta)$ can be identified with the group of automorphisms of a certain graph with additional structure encoding the ``combinatorics'' of gluing.
\end{abstract}


\maketitle

\section{Introduction}

Let $\stripSurf$ be a non-compact two-dimensional manifold and $\Partition$ be a one-dimensional foliation on $\stripSurf$ such that each leaf $\omega$ of $\Partition$ is homeomorphic to $\bR$ and is a closed subset of $\stripSurf$.
These foliations on the plane often appear as level-sets of pseudoharmonic functions and from that point of view they were studied by W.~Kaplan~\cite{Kaplan:DJM:1940}, \cite{Kaplan:DJM:1941}, W.~Boothby~\cite{Boothby:AJM_1:1951}, \cite{Boothby:AJM_2:1951}, M.~Morse and J.~Jenkins~\cite{JenkinsMorse:AJM:1952}, M.~Morse~\cite{Morse:FM:1952} and others.

In particular, Kaplan proved that for every such a foliation there exists at most countably many leaves $\{\omega_i\}_{i\in J}$ such that for every connected component $\strip$ of $\bR^2\setminus\{\omega_i\}_{i\in J}$ one can find a homeomorphism $\phi:\strip\to\bR^2 \times (0,1)$ sending the leaves in $\strip$ onto horizontal lines $\bR\times\{t\}$, $i\in(0,1)$.
However his construction was not \textit{canonical}, as he tried to minimize the total number of strips, and for that reason the closure $\overline{\strip}$ may have very complicated structure.
For instance the above homeomorphism $\phi$ not always extends to an embedding of $\overline{\strip}$ into $\bR\times[0,1]$.

In \cite[Theorem~1.8]{MaksymenkoPolulyakh:MFAT:2016} the first two authors gave sufficient conditions for a one-dimensional foliation on a non-compact surface to have a similar \myemph{striped} structure, and proposed a certain \textit{canonical} decomposition into strips whose closures homeomorphic to open subsets of $\bR\times[0,1]$.

Also, in~\cite{MaksymenkoPolulyakh:PGC:2015} the same authors considered arbitrary foliated surfaces $(\stripSurf,\Partition)$ glued from strips in the above way and studied the homotopy properties of the group of homeomorphisms $\Homeo(\Partition)$ of $\stripSurf$ mapping leaves of the foliation $\Partition$ into leaves. 
They proved that except for few cases the identity path component $\Homeo_0(\Partition)$ of $\Homeo(\Partition)$ is contractible.
The principal technical assumption in~\cite{MaksymenkoPolulyakh:PGC:2015} was that the gluing maps between boundary intervals of strips must be \textit{affine}.

The quotient $\pi_0\Homeo(\Partition) = \Homeo(\Partition)/\Homeo_0(\Partition)$ is an analogue of a mapping class group for foliated homeomorphisms and we call it the \textit{homeotopy} group of the foliation $\Partition$.
In~\cite{Soroka:MFAT:2016} and~\cite{Soroka:UMJ:2017} the third author studied a special class of so-called ``rooted tree like'' striped surfaces, completely described algebraic structure of homeotopy groups of their foliations, and also related those groups with the homeotopy groups of the space of leaves $\stripSurf/\Partition$.

The aim of the present paper is to extend the results of~\cite{MaksymenkoPolulyakh:PGC:2015} to arbitrary ``striped'' surfaces and compute the corresponding homeotopy groups.
Namely, we show that $\pi_0\Homeo(\Partition)$ is isomorphic to a group of automorphism of a certain graph with additional structure, see Theorem~\ref{th:HZ_AutG}.
In particular, these results hold for all foliations considered in~\cite{MaksymenkoPolulyakh:MFAT:2016}, \cite{Soroka:MFAT:2016}, \cite{Soroka:UMJ:2017}.

{\bf Structure of the paper.}
\S\ref{sect:main_results} contains a list of the principal results of the paper.
First we give a formal definition of a strip and then show in Proposition~\ref{pr:strip_repl_by_model_strip} that up to a foliated homeomorphism it can be replaced by a \textit{model} strip having better disposition of boundary intervals.

Next, we characterize homeomorphisms between boundaries of strips which extend to foliated homeomorphisms between strips, see Theorem~\ref{th:homeomorphism_of_strips}.
\S\ref{sect:proof:pr:strip_repl_by_model_strip} and \S\ref{sect:proof:th:homeomorphism_of_strips} are devoted to proofs of those results.

In \S\ref{sect:stripped_atlas} we introduce a notion of a \textit{striped} atlas on a surface $\stripSurf$, being a decomposition into strips glued along boundary intervals, and prove that gluing homeomorphisms can be made affine, see Theorem~\ref{th:all_striped_surf_are_affine}.

Further, in \S\ref{sect:graph_of_a_striped_atlas}, we associate to each striped atlas a certain \textit{graph} $\Gr$ which encodes the ``combinatorics'' of gluing strips, and relate automorphisms of $\Gr$ with self-equivalences of the corresponding atlas, see Theorem~\ref{th:iso_atlas_iso_graph}.

\S\ref{sect:charact_str_prop} establishes relationships between distinct properties of foliated surfaces considered in~\cite{MaksymenkoPolulyakh:PGC:2015}, \cite{MaksymenkoPolulyakh:MFAT:2016}, and \cite{MaksymenkoPolulyakh:PGC:2016}, see Theorem~\ref{th:rel_between_conditions}.

Finally, in \S\ref{sect:homeotopy_group} we consider the group $\HS$ of homeomorphisms of the foliation $\Partition$ and deduce from~\cite{MaksymenkoPolulyakh:PGC:2015} and results of previous sections that the homeotopy group $\pi_0\Homeo(\Partition)$ is isomorphic with the group of automorphisms of the graph associated to some special striped atlas of $\stripSurf$.

\section{Model strips}\label{sect:main_results}
Let $\stripSurf$ be a two-dimensional topological manifold.
A \myemph{foliated chart of dimension $1$} on $\stripSurf$ is a pair $(\Usp, \chartMap)$, where $\Usp \subset \stripSurf$ is an open subset and $\chartMap : \Usp \to (a,b) \times \Btrans$ is a homeomorphism with $\Btrans$ being an open subset of $[0,+\infty)$.
The set $P_y = \chartMap^{-1}\bigl((a,b) \times \{y\}\bigr)$, $y \in \Btrans$, is then called a \myemph{plaque} of this foliated chart.

Suppose $\Partition = \{ \leaf_\alpha \mid \alpha \in A \}$ is a partition of $\stripSurf$ into path connected subsets and there exists an atlas $\mathcal{A} = \{\Usp_i, \chartMap_i\}_{i\in\Lambda}$ of foliated charts of dimension $1$ on $\stripSurf$ such that for each $\alpha \in A$ and each $i \in \Lambda$ every path component of a set $\leaf_\alpha \cap \Usp_i$ is a plaque.
Then $\Partition$ is said to be a \myemph{one-dimensional foliation} on $\stripSurf$ and $\{\Usp_i, \chartMap_i\}_{i\in\Lambda}$ is called a \myemph{foliated atlas associated to $\Partition$}.
Every $\leaf_\alpha$ is then a \myemph{leaf} of the foliation $\Partition$ and the pair $(\stripSurf, \Partition)$ is a \myemph{foliated surface}.

Let $(\stripSurf_1, \Partition_1)$ and $(\stripSurf_2, \Partition_2)$ be two foliated surfaces.
Then a homeomorphism $h:\stripSurf_1 \to \stripSurf_2$ is said to be \myemph{foliated} if for each leaf $\leaf \in\Partition_1$ its image, $h(\leaf)$, is a leaf of $\Partition_2$.

\begin{definition}\label{def:strip}
A subset $\strip \subset \bR^2$ will be called a \myemph{strip} if 
\begin{enumerate}[label=(\roman*)]
\item\label{enum:strip:contains_open_strip}
$\bR \times (u,v) \ \subset \ \strip \ \subset \  \bR \times [u,v]$;
\item\label{enum:strip:open}
$\strip$ is open in the topology of $\bR \times [u,v]$
\end{enumerate}
for some $u<v\in\bR$.
Denote
\begin{align*}
	\partial_{-}\strip &:= \strip \ \cap \ \bR \times \lbrace u \rbrace, &
	\partial_{+}\strip &:= \strip \ \cap \ \bR \times \lbrace v \rbrace, \\
	\partial\strip &:= \partial_{-}\strip \ \cup \ \partial_{+}\strip, &
	\Int\strip &:= \bR\times(u,v).
\end{align*}
We will call $\partial\strip$ the \myemph{boundary} of $\strip$, while $\partial_{-}\strip$ and $\partial_{+}\strip$ will be the \myemph{sides} of $\strip$.
It follows that $\partial\strip$ is an open subset of $\bR\times\{u,v\}$, and so it is a disjoint union of at most countably many open (possibly unbounded) intervals.

If, in addition to~\ref{enum:strip:contains_open_strip} and~\ref{enum:strip:open}, the following conditions hold:
\begin{enumerate}[label=(\roman*), resume]
\item\label{enum:strip:bounded_intervals}
every connected component of $\partial\strip$ is a \myemph{bounded} interval,
\item\label{enum:strip:disjoint_closures}
the closures of boundary intervals of $\partial\strip$ in $\bR\times[u,v]$ are mutually disjoint,
\end{enumerate}
then $\strip$ will be called a \myemph{model strip}.

Evidently, each strip $\strip$ possesses an oriented one-dimensional foliation into horizontal lines $\bR \times t$, $t \in (u,v)$ and boundary intervals of $\partial\strip$.
We will call that foliation \myemph{canonical}.
\end{definition}

The following statement allows to reduce any strip to a technically more convenient form.
It will be proved in Section~\ref{sect:proof:pr:strip_repl_by_model_strip}.
\begin{proposition}\label{pr:strip_repl_by_model_strip}
Each strip is foliated homeomorphic to a model strip.
\end{proposition}

\subsection*{Monotone homeomorphisms of $\partial\strip$.}
Notice that the boundary of a strip can be regarded as a partially ordered set being a disjoint union of two linearly ordered sets $\partial_{-}\strip$ and $\partial_{+}\strip$ that are incomparable with each other.
In other words, for $(a,x), (b,y) \in \partial\strip$ we assume that $(a,x)<(b,y)$ if and only if $a<b$ and $x=y$.

More generally, let $A,B \subset \partial\strip$ be two subsets.
Then we say that $A<B$ if and only if $a<b$ for all $a\in A$ and $b\in B$.
In particular, this gives a linear order on the boundary intervals of $\partial_{-}\strip$ and $\partial_{+}\strip$.
Thus, if $I_{\alpha} = (a,b) \times\{x\}$ and $I_{\beta} = (c,d) \times\{y\}$ are boundary intervals of $\partial\strip$ with $x,y\in\{u,v\}$, then $I_{\alpha} < I_{\beta}$ if and only if $x=y$ and $b < c$.

Now let $\strip_1$ and $\strip_2$ be two strips, $A\subset\partial\strip_1$ and $B\subset\partial\strip_2$ be subsets, and $\dif: A \to B$ be a bijection.
We will say that $\dif$ \myemph{preserves} (resp. \myemph{reverses}) order whenever for any $a,a'\in A$ we have that $a<a'$ if and only if $\dif(a)<\dif(a')$ (resp. $\dif(a)>\dif(a')$). 
In either of these cases $\dif$ is said to be \myemph{monotone}.

Evidently, if $\dif:\strip_1\to\strip_2$ is a foliated homeomorhism between two strips, then its restriction $\dif|_{\partial\strip_1}: \partial\strip_1 \to \partial\strip_2$ is monotone.
The following statement is a converse to the latter observation. 
It will be proved in Section~\ref{sect:proof:th:homeomorphism_of_strips}.

\begin{theorem}\label{th:homeomorphism_of_strips}
Every monotone homeomorphism $\dif: \partial \strip_1 \rightarrow \partial \strip_2$ between boundaries of two strips $\strip_1$ and $\strip_2$ extends to a foliated homeomorphism $\hat{\dif}: \strip_1 \rightarrow \strip_2$.
\end{theorem}

\section{Proof of Proposition~\ref{pr:strip_repl_by_model_strip}}\label{sect:proof:pr:strip_repl_by_model_strip}
\begin{lemma}\label{lm:half_strip_repl_by_model_half_strip}
Let $\strip$ be a half strip with $\bR\times[0,1) \subset \strip \subset \bR\times[0,1]$.
Then there exists a half strip $\strip'$ and foliated homeomorphism $\dif:\strip\to\strip'$ such that 
\begin{itemize}[leftmargin=2em]
\item 
$\bR\times[0,1) \ \subset \ \strip' \ \subset \ \strip \  \subset \ \bR\times[0,1]$;
\item 
the closures of boundary intervals of $\partial_{+}\strip'$ are bounded in $\bR^2$ and mutually disjoint;
\item 
$\dif$ is fixed on $\bR\times 0$;
\item 
$\dif$ preserves the second coordinate, and therefore is foliated.
\end{itemize}
\end{lemma}

Assuming lemma is true let us deduce Proposition~\ref{pr:strip_repl_by_model_strip}.
Let $\strip$ be a strip with $\Int(\strip)=\bR\times(-1,1)$.
Consider two half strips 
\begin{align*}
A\, &= \, \strip \, \cap \, \bigl(\bR\times[-1,0]\bigr), &
B\, &= \, \strip \, \cap \, \bigl(\bR\times[0,1]\bigr).
\end{align*}
Then by Lemma~\ref{lm:half_strip_repl_by_model_half_strip} one can find two half strips $A'$ and $B'$ and foliated homeomorphisms $f:A \to A'$ and $g:B \to B'$ such that 
\begin{itemize}[leftmargin=2em]
\item 
$\bR\times(-1,0] \ \subset \ A'  \ \subset \ A \  \subset \ \bR\times[-1,0]$;
\item 
$\bR\times[0,1) \ \subset \ B'  \ \subset \ B \  \subset \ \bR\times[0,1]$;
\item 
the closures of boundary intervals of $\partial_{-} A'$ and $\partial_{+} B'$ are bounded in $\bR^2$ and mutually disjoint;
\item 
$f$ and $g$ are fixed on $\bR\times 0$ and preserve second coordinate.
\end{itemize}
Then $\strip' = A' \cup B'$ is a model strip with $\bR\times[0,1)  \subset  \strip'  \subset  \strip$,
and a foliated homeomorphism $\dif:\strip\to\strip'$ can be given by the formula: $\dif|_{A'} = f$ and $\dif|_{B'} = g$.
This proves Proposition~\ref{pr:strip_repl_by_model_strip} modulo Lemma~\ref{lm:half_strip_repl_by_model_half_strip}.

\begin{proof}[Proof of Lemma~\ref{lm:half_strip_repl_by_model_half_strip}]
a) First we will show how to make closures of boundary intervals of $\partial_{+}\strip'$ to be bounded though not necessarily disjoint.
Fix any $a<b \in \bR$ and consider the following half strip:
\[
T \ = \ \bR\times[0,1)  \ \cup \ (a,b) \times \{1\}.
\]
Then by~\cite[Lemma~3.2]{MaksymenkoPolulyakh:PGC:2015} there exists a homeomorphism $\dif:\bR\times[0,1] \to T$ preserving second coordinate and fixed on $\bR\times 0$.
Hence $\strip' = \dif(\strip)$ is a half strip with $\partial \strip' \subset \partial T = (a,b)\times 1$.

b) To simplify the notation replace $\strip$ with $\strip'$ and assume that boundary intervals of $\partial_{+}\strip'$ are bounded in $\bR^2$.
We should make their closures mutually disjoint.

Consider the following subset of $\bR^{7}$:
\[
A=\left\lbrace\, (a,u,b,c,v,d,t) \, \mid \, a<u<b, \ c<v<d, \ t\in [a,b] \, \right\rbrace
\]
and define the function $\gamma: A \rightarrow \bR$ by
\[
\gamma(a,u,b,c,v,d,t) =
\begin{cases}
\dfrac{t-a}{u-a}(v-c)+c, & t \in [a;u], \\ 
\dfrac{t-u}{b-u}(d-v)+v, \phantom{\dfrac{A^{A^A}}{A^A}} & t \in [u;b].	
\end{cases}
\]
Then $\gamma$ is continuous, and for any combination of the first six parameters $a,u,b,c,v,d$ the map $t \mapsto \gamma(a,u,b,c,v,d,t)$ homeomorphically maps the segment $[a,b]$ onto $[c,d]$ so that $u$ is sent to $v$.

\begin{sublemma}\label{lm:triangles}
Let  $T$ be a closed triangle in the plane $\bR^2$ with vertices $A(x_a, y_a)$, $B(x_b,y_b)$, $O(x_o,y_o)$ such that $y_a=y_b>y_o$, and  $C$ be a point on the open interval $(A,B)$.
Let also 
\begin{align*}
 T'&=T\setminus \left\lbrace A,B\right\rbrace, & 
 T''& =T \setminus \left\lbrace A \cup [B,C] \right\rbrace.
\end{align*}
Then there exists a homeomorphism $f: T' \rightarrow T''$ preserving second coordinate.
In particular, $f$ is fixed on the sides $(A,O]$, $[O,B)$ and maps the interval $(A,B)$ onto $(A,C)$.
\end{sublemma}
\begin{proof}
Not loosing generality one can assume that 
$A(-1,0)$, $B(1,1)$, $O(0,0)$ and $C(0.5,1)$.
Then $f$ can be given by the formula:
\[
f(x,y)=
\begin{cases}
\left( \gamma(-y,y^2,y,-y,0,y,x), y\right), & 0 \leqslant y <1,\\
\left(\frac{x+1}{2}-1,1 \right), & y=1. 
\end{cases}
\]
Evidently, $f$ maps the curve $x=y^2$ on the segment of the line $x=0$.
Moreover, we have that $f[A_y, Q_y] = [A_y, C_y]$, and $f[Q_y,B_y]=[C_y, B_y]$, see 
Figure~\ref{fig:triangles},
\begin{figure}[h]
\center{\includegraphics[height=2.5cm]{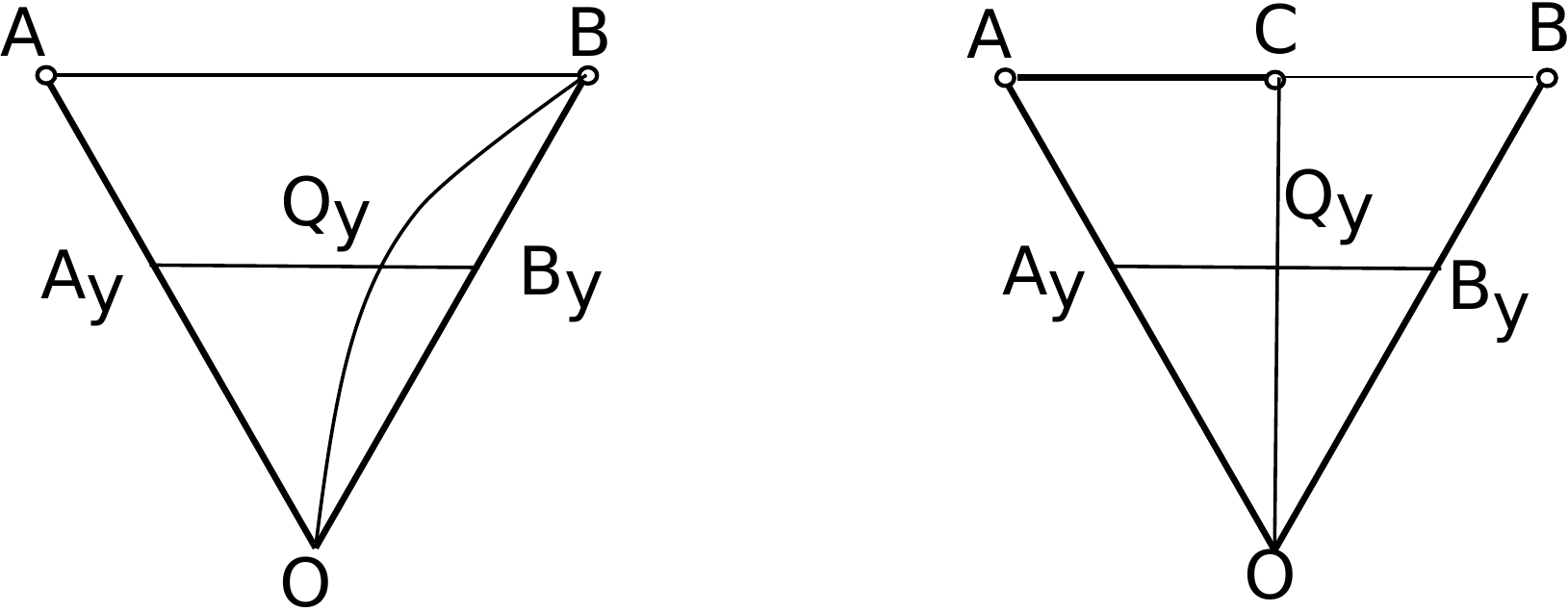}}
\caption{Triangles $T'$ and $T''$}
\label{fig:triangles}
\end{figure}
\end{proof}

Returning back to the proof of Lemma~\ref{lm:half_strip_repl_by_model_half_strip} assume that 
\[ \partial_{+} \strip = \bigsqcup\limits_{i=1}^{N} (a_i, b_i) \times \{1\},\]
where $N$ is either a finite number of $+\infty$.
Fix a strictly monotone sequence $\{\leveld{j}\}^{\infty}_{j=1} \subset [-1,1]$ such that $\lim\limits_{j \to \infty} = 1$ and for each $i$ define the triangle $T_i$ with vertices $A_i(a_i,1)$, $B_i(b_i,1)$, $O_i(\frac{a_i+b_i}{2},u_i)$, and let $C(\frac{a_i+b_i}{2},1)$.
Define also the following model half strips:
\[
\strip_i = \strip \ \setminus \ \bigcup_{j\leq i} \ [C_j,B_j) \times \{1\}, 
\]
and put
\[\strip_0 = \strip, \qquad 
\strip' = \strip \ \setminus \ \bigcup_{i=1}^{N} \ [C_i,B_i) \times \{1\} = \bigcap_{i=1}^{N} \strip_i. \]
Then the closures of the boundary intervals of $\strip'$ are mutually disjoint.

Denote $T'_i = T_i \setminus \{A_i, B_i\}$ and $T''_i = T_i \setminus \{A_i \cup [B_i, C_i]\}$.
Then by Lemma~\ref{lm:triangles} there exists a homeomorphism $f_i: T'_i \rightarrow T''_i$ preserving second coordinate and being identity on the sides $(A_i,O_i]$ and $[O_i,B_i)$.
Therefore $f_i$ extends by the identity to a homeomorphism $f_i: \strip_{i-1} \to \strip_{i}$.

Then a foliated homeomorphism $\varphi: \strip \rightarrow \strip^{'}$ can be defined as the composition of all $f_i$:
\[
\varphi = \cdots\circ f_{i+1} \circ f_{i} \circ\cdots \circ f_2 \circ f_{1}: 
\strip=\strip_0 \xrightarrow{f_1} \strip_1 \xrightarrow{f_2} 
\cdots  \xrightarrow{f_{i}} \strip_i \xrightarrow{f_{i+1}} \cdots \strip'.
\]
If $N$ is finite, $\varphi$ is well-defined.

For infinite $N$ one should check that for each point $z\in\strip$ the sequence 
\[ \bigl\{ \ g_j = f_j \circ \cdots \circ f_2 \circ f_{1}: \strip \to \strip_j \subset \strip \ \bigr\}_{j\in\bN}\] of embeddings ``stabilizes'' on some neighborhood $\Usp_z$ of $z$, that is $g_j = g_{j+1}$ on $\Usp$ for all sufficiently large $j$.

So, let $z=(x,y)\in\strip$.
If $y<1$, then there exists $i$ such that $y<u_i$.
Let $\Usp=\bR \times (0,u_i)$.
Then $f_j(\Usp) = \Usp$ for all $j$.
Moreover, $f_j$ is fixed on $\Usp$ for all $j>i$.
Therefore $g_j = g_i$ on $\Usp$ for all $j>i$.

Suppose $y=1$, so $(x,y) \in (a_i,b_i)\times \{1\}$ for some $i$, then each $f_j$ with $j\not=i$ is fixed on the triangle $T'_i$.
Then $\Usp = T'_i \setminus \bigl( (A_i,O_i] \cup [O_i,B_i) \bigr)$ is an open neighborhood of $(x,y)$ in $\strip$, and $g_j = g_i = f_i$ on $\Usp$ for $j>i$.
Thus $\varphi$ is a homeomorphism.
Lemma~\ref{lm:half_strip_repl_by_model_half_strip} is completed.
\end{proof}

\section{Proof of Theorem~\ref{th:homeomorphism_of_strips}}\label{sect:proof:th:homeomorphism_of_strips}
Let $\strip_1$ and $\strip_2$ be two strips and $\dif: \partial \strip_1 \to \partial \strip_2$ a monotone homeomorphism.
We should prove that $\dif$ extends to a foliated homeomorphism between $\strip_1$ and $\strip_2$.

If $\partial\strip_1$, and so $\partial \strip_2$, are empty, then any foliated homeomorphism between $\strip_1$ and $\strip_2$ can be regarded as an extension of $\dif$.
Therefore we will suppose that $\partial\strip_1\not=\varnothing$.
Not loosing generality one can also assume that $\dif(\partial_{-}\strip_1)=\partial_{-}\strip_2$, $\dif(\partial_{+}\strip_1)=\partial_{+}\strip_2$, and the restrictions $\dif|_{\partial_{-}\strip_1}$ and $\dif|_{\partial_{+}\strip_1}$ are increasing.

\subsection*{Case 1.} 
Suppose that both $\strip_1$ and $\strip_2$ are half strips such that
\begin{gather*}
\partial_{+} \strip_1 = \bigsqcup_{i=1}^{K} \bdX_{i}  \times \{1\}, 
\qquad 
\partial_{+} \strip_2 = \bigsqcup_{i=1}^{K} \bdX'_{i}  \times \{1\}, \\
\partial_{-} \strip_1=\partial_{-} \strip_2=\bR\times\{0\},
\end{gather*}
where $K \in \{0,1,\ldots, +\infty\}$, each $\bdX_{i}\times\{1\}$ is a boundary interval of $\partial_{+}\strip_1$, $\bdX'_{i}$ is a boundary interval of $\partial_{+}\strip_2$, the closures $\overline{\bdX_{i}}$ and $\overline{\bdX'_{i}}$ are bounded, 
\begin{align}\label{equ:case1:half_strips_disjoint_closures}
\overline{\bdX_{i}} \cap \overline{\bdX_{j}} &= \varnothing, &
\overline{\bdX'_{i}} \cap \overline{\bdX'_{j}} &= \varnothing,
\end{align}
for all $i\not= j$, and $\dif(\bdX_{i}\times\{1\}) = \bdX'_{i}\times\{1\}$.

We will extend $\dif$ to a homeomorphism $\dif: \strip_1 \to \strip_2$ preserving second coordinate and fixed on $\bR\times 0$.

Fix an arbitrary strictly increasing sequence $\{\leveld{j}\}^{\infty}_{j=0} \subset [0,1)$ such that $\leveld{0} = 0$ and $\lim\limits_{j\to\infty} \leveld{j} = 1$.
For each $\leveld{j}=0,1,\ldots,\infty$ we will now construct a homeomorphism $\psi_j: \bR \rightarrow \bR$ by the following rule.

Since $\dif(\bR\times0) = \bR\times0$ one can write $\dif(x,0) = (\psi_0(x),0)$ for a unique homeomorphism $\psi_0:\bR\to\bR$.

Further notice that there exists a unique homeomorphism 
\[\overline{h}: \bigsqcup\limits_{i=1}^{K} \bdX_i \to \bigsqcup\limits_{i=1}^{K} \bdX'_i\]
such that 
$\overline{h}(\bdX_{i}) = \bdX'_{i}$.
Then for $j\geq1$ define $\psi_j: \bigsqcup\limits_{i=1}^{j} X_i \to \bigsqcup\limits_{i=1}^{j} X'_i$ by $\psi_j(x) = \overline{h}(x)$.
The assumption that $\dif$ preserves the order of boundary intervals means that $\bdX_i < \bdX_j$ if and only if $\bdX^{'}_i < \bdX^{'}_j$, $i,j \in 1,\ldots, K$.
Hence one can apply the following Lemma~\ref{lem:homeomorphism_of_line} to extend $\psi_j$ to a homeomorphism $\psi_j: \bR \rightarrow \bR$.

\begin{lemma}\label{lem:homeomorphism_of_line}
Let $\alpha=\{\bdX_{i}\}^{n}_{i=1}$ and $\beta= \{\bdX'_{i} \}^{n}_{i=1}$ be two families of \myemph{open} segments in $\bR$ having the following properties:
\begin{enumerate}[label={\rm(\arabic*)}, leftmargin=*]
\item\label{enum:lem: homeomorphism of line:disjoint}
the closures $\overline{\bdX_{i}}$ and $\overline{\bdX'_{i}}$ are bounded, and $\overline{\bdX_{i}} \cap \overline{\bdX_{j}} = \overline{\bdX'_{i}} \cap \overline{\bdX'_{j}} =\emptyset$ for all $i\not=j =1,\ldots, n$;
\item\label{enum:lem: homeomorphism of line:similarly_ordered}
$\alpha$ and $\beta$ are ``similarly ordered'', that is $\bdX_i < \bdX_j$ if and only if $\bdX'_i < \bdX'_j$ for all $i\not=j =1,\ldots, n$.
\end{enumerate}
Suppose also that for each $i=1,\ldots,n$ we have an orientation preserving homeomorphism $\psi_i: \bdX_i \rightarrow \bdX'_i$.
Then there is a homeomorphism $\psi: \bR \rightarrow \bR$ such that $\psi|_{\bdX_i }=\psi_{i}$.
\end{lemma}
\begin{proof}
Due to assumptions on $\alpha$ and $\beta$ one can renumber the elements in these families and assume that $\bdX_{i}= (a_i,b_i)$ and $\bdX_{i}^{'}= (c_i,d_i)$ for some $a_i,b_i,c_i,b_i\in\bR$ such that
\begin{align*}
&a_1< b_1< a_2 < b_2 < \ldots <a_n< b_n, & \ \
&c_1< d_1< c_2 < d_2 < \ldots <c_n< d_n.
\end{align*}
Then the homeomorphism $\psi$ can be given by the formula:
\[
\psi(x) =
\begin{cases}
x-a_1+ c_1, & x \in (-\infty, a_1],\\
\psi_i(x), &  x \in (a_i,b_{i}), \ i=1,\ldots,n,\\
\frac{c_{i+1}-d_i}{a_{i+1}-b_i}(x-b_i)+d_i, & x \in [b_i,a_{i+1}], \ i=1,\ldots,n-1, \\
x-b_n+ d_n, & x \in [b_n, +\infty).
\end{cases}
\]
Lemma~\ref{lem:homeomorphism_of_line} is proved.
\end{proof}

Now define $\dif:\strip_1\to\strip_2$ by the formula:
\begin{equation}\label{equ:formula_for_h}
\dif(x,y)=
\begin{cases}
\bigl( \psi_j(x), y\bigr), & y=u_j, \\ & 0\leq j \leq K, \\
\bigl( \varepsilon_j(y) \psi_j(x)+(1-\varepsilon_j(y))\psi_{j+1}(x), \ y \bigr), & y \in (u_j,u_{j+1}), \\ & 0 \leq j \leq K-1,\\
\end{cases}
\end{equation}
where $\varepsilon_j(y)=\frac{u_{j+1}-y}{u_{j+1}-u_j}$.

Evidently, $\dif$ is bijective, preserves the second coordinate and homeomorphically maps $\strip_{1} \setminus \partial_{+} \strip_{1}$ onto $\strip_{2} \setminus \partial_{+} \strip_{2}$.

It remains to check that $\dif$ is a homeomorphism.
Since $\psi_j = \psi_{j+1} = \overline{h}$ on $\bdX_i$ for $i\leq j$, it follows from the second line in~\eqref{equ:formula_for_h} that
\[
\dif(x,y)=(\overline{\dif}(x),y).
\]
for all $(x,y) \in \bdX_i \times (\leveld{i}, 1]$.
Therefore $\dif$ homeomorphically maps the open set $\bdX_i \times (\leveld{i}, 1]$ of $\strip_1$ onto the open set $\bdX'_i \times (\leveld{i}, 1]$ of $\strip_2$.
Since the family 
\[
\{ \bdX_i \times (\leveld{i}, 1] \}_{i=1}^{K} \ \cup \ \{  \strip_{1} \setminus \partial_{+} \strip_{1} \}
\]
constitutes an open covering of $\strip_1$, it follows that $\dif$ is a homeomorphism.

\subsection*{Case 2.}
Suppose $\strip_1$ and $\strip_2$ are arbitrary half strips.

One can assume that $\bR\times[0,1) \subset \strip_i \subset \bR\times[0,1]$, $i=1,2$.
Then by Lemma~\ref{lm:half_strip_repl_by_model_half_strip} one can find a half strip $\strip'_i$ and a homeomorphism $\phi_i:\strip_i \to \strip'_i$ such that  
\begin{itemize}
\item 
$\bR\times[0,1) \subset \strip'_i \subset \strip_i \subset \bR\times[0,1]$,
\item 
the closures of boundary intervals in $\partial_{+}\strip'_i$ are bounded and mutually disjoint;
\item 
$\phi_i$ is fixed on $\bR\times 0$ and preserves the second coordinate.
\end{itemize}
Hence the composition $\dif' = \phi_2 \circ \dif \circ \phi_1^{-1}: \partial\strip'_1 \to \partial\strip'_2$ is a homeomorphism preserving order and orientations of boundary intervals and coincides with $\dif$ on $\bR\times 0$.
Therefore, by Case~1, it extends to a foliated homeomorphism $\dif':\strip'_1 \to \strip'_2$.
Hence $\phi_2^{-1} \circ \dif' \circ \phi_1:\strip_1 \to \strip_2$ is the required extension of $\dif$.

\subsection*{Case 3.}
Consider the general case when $\strip_1$ and $\strip_2$ are strips.
Not loosing generality one can assume that 
\[
\bR\times(-1,1) \subset \strip_i \subset \bR\times[-1,1]
\]
for $i=1,2$.
Similarly to the proof of Proposition~\ref{pr:strip_repl_by_model_strip} consider two half strips 
\begin{align*}
 A_i &= \strip_i \cap \bR\times[-1,0], & B_i &= \strip_i \cap \bR\times[0,1]. 
\end{align*}
Evidently,
\begin{align*}
\partial_{-} A_i &= \partial_{-} \strip_i,  &
\partial_{+} A_i = \partial_{-} B_i &= \bR\times 0, &
\partial_{+} B_i &= \partial_{+} \strip_i.
\end{align*}
Define two homeomorphisms $f:\partial A_1 \to \partial A_2$  and $g:\partial B_1 \to \partial B_2$ by the rule:
\begin{align*}
f|_{\partial_{-} A_1} &= \dif|_{\partial_{-} A_1}, &
f|_{\partial_{+} A_1} &= g|_{\partial_{-} B_1} = \id_{\bR\times 0}, &
f|_{\partial_{+} B_1} &= \dif|_{\partial_{+} \strip_1}.
\end{align*}
Then, by Case~2, $f$ and $g$ extend to foliated homeomorphisms $f:A_1\to A_2$ and $g:B_1\to B_2$.
Hence a required extension $\dif:\strip_1\to\strip_2$ of $\dif$ can be given by the formula: $\dif|_{A'} = f$ and $\dif|_{B'} = g$.
Theorem~\ref{th:homeomorphism_of_strips} is completed.
\qed

\section{Striped atlas}\label{sect:stripped_atlas}
Let $\stripSurf$ be a two-dimensional topological manifold.
\begin{definition}
A \myemph{striped atlas} on $\stripSurf$ is a map $\qmap: \preStripSurf \to \stripSurf$ having the following properties:
\begin{enumerate}[leftmargin=*, label=(\arabic*)]
\item 
$\preStripSurf = \bigsqcup \limits_{\stInd \in \StInd} \strip_{\stInd}$ is at most countable family of mutually disjoint strips;

\item  
$\qmap$ is a \myemph{quotient} map, which means that it is continuous, surjective, and has the property that a subset $\Usp\subset\stripSurf$ is open if and only if $\qmap^{-1}(\Usp) \cap \strip_{\stInd}$ is open in $\strip_{\stInd}$ for each $\stInd\in\StInd$;

\item 
there exist two disjoint families $\mathcal{X} = \{\bdX_\bdGlueInd\}_{\bdGlueInd\in\BdGlueInd}$ and $\mathcal{Y} = \{\bdY_\bdGlueInd\}_{\bdGlueInd\in\BdGlueInd}$ of mutually disjoint boundary intervals of $\preStripSurf$ enumerated by the same set of indexes $\BdGlueInd$ such that 
\begin{enumerate}[leftmargin=*, label=(\alph*)]
\item
$\qmap$ is injective on $\preStripSurf \setminus (\mathcal{X} \cup \mathcal{Y})$;
\item
$\qmap(\bdX_\bdGlueInd) = \qmap(\bdY_\bdGlueInd)$ for each $\bdGlueInd\in\BdGlueInd$;
\item\label{enum:striped_atlas:XY:embeddings}
the restrictions $\qmap|_{\bdX_\bdGlueInd}: \bdX_\bdGlueInd \to \qmap(\bdX_\bdGlueInd)$ and 
$\qmap|_{\bdY_\bdGlueInd}: \bdY_\bdGlueInd \to \qmap(\bdY_\bdGlueInd)$ are embeddings with closed images;
\end{enumerate}
\end{enumerate}
\end{definition}

\begin{definition}
A surface $\stripSurf$ admitting a striped atlas will be called a \myemph{striped} surface.
\end{definition}

Notice that a striped surface $\stripSurf$ is a non-compact two-dimensional manifold which can be non-connected and non-orientable, and each of its boundary component is an open interval.

Moreover, $\stripSurf$ admits a one-dimensional foliation obtained from canonical foliations on the corresponding model strips $S_{\lambda}$.
We will call it the \myemph{canonical} foliation associated to the striped atlas $\qmap$ and denote by $\Partition$.
Evidently, each leaf of $\Partition$ is a homeomorphic image of $\bR$ and is also a closed subset of $\stripSurf$.

\begin{definition}\label{def:agreeing_atlas_and_foliation}
We will say that a foliated surface $(\stripSurf,\Partition)$ admits a \myemph{striped structure} if there exists a striped atlas $\qmap:\preStripSurf\to\stripSurf$ which maps each leaf of the canonical foliation of each strip in $\preStripSurf$ onto some leaf of $\Partition$.
\end{definition}

\begin{remark}
Due to~\ref{enum:striped_atlas:XY:embeddings} for each $\bdGlueInd \in \BdGlueInd$ one get the following ``gluing'' homeomorphism $\phi_{\bdGlueInd}: \bdY_{\bdGlueInd} \to \bdX_{\bdGlueInd}$ defined by 
\begin{equation}\label{equ:gluing_map}
\phi_{\bdGlueInd} = \bigl(\qmap|_{\bdX_{\bdGlueInd}}\bigr)^{-1} \circ  \qmap|_{\bdY_{\bdGlueInd}}.
\end{equation}
Therefore one can think that a striped surface is obtained from a family of model strips by gluing them along certain boundary intervals by homeomorphisms $\phi_{\bdGlueInd}$.
It is allowed that two strips are glued along more than one pair of boundary components.
Moreover, one may glue together boundary components of the same strip $\strip$.
\end{remark}

\begin{definition}\label{def:strip_surf_equiv}
Two striped atlases $\qmap : \preStripSurf \to \stripSurf$ and $\qmap' : \preStripSurf' \to \stripSurf'$ on striped surfaces $\stripSurf$ and $\stripSurf'$ will be called \myemph{equivalent} if there exist two foliated homeomorphisms $h : \preStripSurf \to \preStripSurf'$ and $k : \stripSurf \to \stripSurf'$ making commutative the following diagram:
\begin{equation}\label{eq:strip_surf_equiv}
\begin{CD}
\preStripSurf @>{h}>> \preStripSurf'\\
@V{\qmap}VV @VV{\qmap'}V \\
\stripSurf @>>{k}> \stripSurf'
\end{CD}
\end{equation}
\end{definition}

Turning back to the definition of a striped surface notice that for each $\bdGlueInd \in \BdGlueInd$ the intervals $\bdX_{\bdGlueInd}, \bdY_{\bdGlueInd}$ are horizontal, and so
\begin{align*}
\bdX_{\bdGlueInd} &=(a,b) \times \{x_\bdGlueInd\}, &
\bdY_{\bdGlueInd} &= (c,d) \times \{y_\bdGlueInd\}
\end{align*}
for some $x_\bdGlueInd, y_\bdGlueInd \in \{u, v\}$ and $a,b, c,d \in \bR\cup \{\pm\infty\}$ with $a<b$ and $c<d$.
Hence $\phi_\bdGlueInd : \bdY_\bdGlueInd \to \bdX_\bdGlueInd$ can be written as follows:
\begin{equation}\label{eq:gluing_map_coord}
\phi_\bdGlueInd(s, y_\bdGlueInd) = (\psi_\bdGlueInd(s), x_\bdGlueInd), \quad s\in (c,d),
\end{equation}
where $\psi_\bdGlueInd:(c,d)\to(a,b)$ is a certain homeomorphism.

\begin{remark}\label{rem:affine_gluing}
Notice that if $a<b$ and $c<d$, then there exist exactly two \textit{affine} homeomorphisms $\psi^{+}, \psi^{-}: (c,d) \to (a,b)$ given by 
\begin{align}\label{eq:affine_gluing}
\psi^{+}(t)&= \frac{b-a}{d-c}\bigl(t - c \bigr) + a,
&
\psi^{-}(t)&= \frac{a-b}{d-c}\bigl(t - c \bigr) + b,
\end{align}
for $t\in(a,b)$.
Evidently, $\psi^{+}$ preserves the orientation and $\psi^{-}$ reverses it.
\end{remark}

\begin{definition}
A striped atlas $\qmap:\preStripSurf\to\stripSurf$ will be called \emph{affine} if the following two conditions hold:
\begin{enumerate}[label=$(\alph*)$, leftmargin=*]
\item 
$\preStripSurf$ consists of model strips only;
\item 
each gluing map $\phi_{\bdGlueInd}: \bdY_{\bdGlueInd} \to \bdX_{\bdGlueInd}$, $\bdGlueInd \in \BdGlueInd$, is \emph{affine}, that is the homeomorphism $\psi_\bdGlueInd$ in~\eqref{eq:gluing_map_coord} is given by either of the formulas from~\eqref{eq:affine_gluing}.
\end{enumerate}
\end{definition}

\begin{theorem}\label{th:all_striped_surf_are_affine}
Each striped atlas $\qmap:\preStripSurf=\bigsqcup \limits_{\stInd \in \StInd} \strip_{\stInd}\to\stripSurf$ on a striped surface $\stripSurf$ is equivalent to an affine one.
Moreover, if $\preStripSurf$ consists of model strips only, then there exists a foliated homeomorphism $h:\preStripSurf\to\preStripSurf$ such that the composition $\qmap' = \qmap\circ\dif:\preStripSurf\to\stripSurf$ is an affine atlas.
\end{theorem}
\begin{proof}
\newcommand\aff{\sigma}
First we show that $\qmap$ is equivalent to an atlas consisting of model strips only.
By Proposition~\ref{pr:strip_repl_by_model_strip} for every strip $\strip_{\stInd}$ there exists a model strip $\strip'_{\stInd} \subset \strip_{\stInd}$ and a foliated homeomorphism $\dif_{\stInd}:\strip'_{\stInd}\to \strip_{\stInd}$.
Put $\preStripSurf' = \bigsqcup \limits_{\stInd \in \StInd} \strip'_{\stInd}$ and define a homeomorphism $\dif:\preStripSurf'\to \preStripSurf$ by $\dif|_{\strip'_{\stInd}} = \dif_{\stInd}$, $\stInd\in\StInd$.
Then $\qmap' = \qmap \circ\dif:\preStripSurf' \to \stripSurf$ is an atlas on $\stripSurf$ consisting of model strips and the pair $(\dif, \id_{\stripSurf})$ is an equivalence between $\qmap'$ and $\qmap$.

\medskip

Assume now that each strip $\strip_{\stInd}$ in $\preStripSurf$ is model.
For each $\bdGlueInd \in \BdGlueInd$ let $\aff_{\bdGlueInd}:\bdY_{\bdGlueInd} \to \bdX_{\bdGlueInd}$ be a unique affine homeomorphism preserving or reversing orientation mutually with $\phi_{\bdGlueInd}$.

Let $\strip_{\stInd}$, $\stInd \in \StInd$, be a model strip from $\preStripSurf$ and $\partial\strip_{\stInd}=\mathop{\sqcup}\limits_{\aind\in \Aind} I_{\aind}$ be the family of its boundary intervals.
We will now define a certain homeomorphism $\dif_{\stInd}: \partial\strip_{\stInd} \to \partial \strip_{\stInd}$ preserving each $I_{\aind}$ with its orientation.
If $I_{\aind} = \bdX_{\bdGlueInd}$ for some $\bdGlueInd \in \BdGlueInd$, then we set 
\[
\dif_{\stInd} = \phi_{\bdGlueInd} \circ \sigma_{\bdGlueInd}^{-1}: \bdX_{\bdGlueInd} \to \bdX_{\bdGlueInd},
\]
otherwise put $\dif_{\stInd}$ to be the identity map $\id_{I_{\aind}}$.

Then $\dif_{\stInd}$ satisfies assumptions of Theorem~\ref{th:homeomorphism_of_strips} and therefore extends to a foliated homeomorphism $\dif_{\stInd}: \strip_{\stInd} \to \strip_{\stInd}$.
Hence we get a homeomorphism $\dif:\preStripSurf \to \preStripSurf$ defined by $\dif|_{\strip_{\stInd}} = \dif_{\stInd}$.

Then one easily checks that the map $\qmap' = \qmap\circ\dif: \preStripSurf \to \stripSurf$ is a striped atlas for $\stripSurf$.
Moreover, $\qmap'$ glues the same strips along the same boundary intervals and in the same directions as $\qmap$, but its gluing maps
\[
\phi'_{\bdGlueInd} = \bigl(\qmap'|_{\bdX_{\bdGlueInd}}\bigr)^{-1} \circ  \qmap'|_{\bdY_{\bdGlueInd}}: \bdY_{\bdGlueInd} \to \bdX_{\bdGlueInd}
\]
differs from the ones of $\qmap$.
It follows from the following commutative diagram:
\[
\xymatrix{
\bdY_{\bdGlueInd} 
  \ar[rrrrr]_(.4){\dif|_{\bdY_{\bdGlueInd}} = \id_{\bdY_{\bdGlueInd}} } 
  \ar[d]_{\phi'_{\bdGlueInd}}
  \ar@/^1pc/[rrrrrrr]^{\qmap'|_{\bdY_\bdGlueInd}}
   &&&&&
\bdY_{\bdGlueInd}
  \ar[rr]_{\qmap|_{\bdY_{\bdGlueInd}}} 
  \ar[d]^{\phi_{\bdGlueInd} } 
  \ar[dlllll]^(.3){\sigma_{\bdGlueInd} } 
  &&
\qmap(\bdY_{\bdGlueInd}) 
  \ar@{=}[d] \\
\bdX_{\bdGlueInd} 
  \ar@/_1pc/[rrrrrrr]_{\qmap'|_{\bdX_\bdGlueInd}}
  \ar[rrrrr]^(.7){\dif|_{\bdX_{\bdGlueInd}} = \phi_{\bdGlueInd} \circ\sigma^{-1}_{\bdGlueInd} } &&&&&
\bdX_{\bdGlueInd} 
  \ar[rr]^{\qmap|_{\bdX_\bdGlueInd}} &&
\qmap(\bdX_{\bdGlueInd}) 
}
\]
that $\phi'_{\bdGlueInd} = \sigma_{\bdGlueInd}$ is affine.
Hence $\qmap'$ is an affine striped atlas for $\stripSurf$.
\end{proof}

\section{Graph of a striped atlas}\label{sect:graph_of_a_striped_atlas}
Let $\qmap: \mathop{\sqcup}\limits_{\stInd\in\StInd}\strip_{\stInd} \to \stripSurf$ be a striped atlas on $\stripSurf$.
We will now associate to $\qmap$ a certain graph $\Gr$ which encodes a ``combinatorial'' information about gluing strips via $\qmap$.
It was firstly considered in~\cite{Soroka:MFAT:2016} for a special class of ``rooted tree like'' striped surfaces.
That graph may have multiple edges and loops and also half-open edges.
\begin{enumerate}[leftmargin=*]
\item 
The \textit{vertices} of $\Gr$ are strips of $\mathop{\sqcup}\limits_{\stInd\in\StInd}\strip_{\stInd}$.

\item 
It will be convenient to call each boundary interval $\bdX$ of some strip $\strip_{\stInd}$ a \textit{half-edge incident to the vertex $\strip_{\stInd}$}.
The set of all half edges of $\partial_{\pm}\strip_{\stInd}$ will be denoted by $d_{\pm}(\strip_{\stInd})$.
We also put $d(\strip_{\stInd})=d_{-}(\strip_{\stInd}) \cup d_{+}(\strip_{\stInd})$.

\item
The \textit{edges} of $\Gr$ are of the following two types.
\begin{enumerate}[leftmargin=*]
\item 
If two strips $\strip_1$ and $\strip_2$ are glued along their boundary intervals $\bdX_\bdGlueInd$ and $\bdY_\bdGlueInd$, then we assume that the \textit{vertices} $\strip_1$ and $\strip_2$ of $\Gr$ are connected by an edge $e_{\bdGlueInd}$.
Thus formally, an edge $e_{\bdGlueInd}$ is an unordered pair of half edges $(\bdX_\bdGlueInd, \bdY_\bdGlueInd)$ and will be called a \textit{closed} edge of $\Gr$.

\item
If $\bdX$ is a boundary interval of some strip $\strip_{\stInd}$ which is not glued to any other interval, so it represents a boundary interval of $\stripSurf$, then we assume that $\bdX$ is a half-open edge with one vertex $\strip_{\stInd}$.
\end{enumerate}
\end{enumerate}

We also add to $\Gr$ the information about directions of gluing boundary intervals, and the disposition of boundary intervals along each strip.

For a homeomorphism $f:(a,b) \to (c,d)$ define a number $\ori(f) = +1$ if $f$ preserves orientation and $\ori(f) = -1$ otherwise.
It is evident, that if $g:(c,d) \to (e,f)$ is another homeomorphism, then $\ori(g\circ f) = \ori(g)\cdot \ori(f)$.

\begin{enumerate}[leftmargin=*, resume]
\item
To each closed edge $(\bdX_\bdGlueInd, \bdY_\bdGlueInd)$ corresponding to the gluing of boundary components $\phi_{\bdGlueInd}:\bdY_\bdGlueInd \to \bdX_\bdGlueInd$ we associate the number $\sigma(\bdX_\bdGlueInd, \bdY_\bdGlueInd):=\ori(\phi_{\bdGlueInd})$ and call it the \textit{orientation of gluing}.

\item
Recall that for each strip $\strip_{\stInd}$ the set of its boundary intervals is at most countable partially ordered set being a disjoint union of two linearly ordered subsets corresponding to $\partial_{-}\strip_{\stInd}$ and $\partial_{+}\strip_{\stInd}$ respectively.
Therefore we have a \textit{linear order} on each of the sets $d_{-}(\strip_{\stInd})$ and $d_{+}(\strip_{\stInd})$ of all half-edges incident to the \textit{vertex} $\strip_{\stInd}$ of $\Gr$.
\end{enumerate}

Thus, ``very formally'', a graph of a striped atlas is the following object
\[
\Gr = (\StInd, H, \xi, \sigma)
\]
where
\begin{itemize}[leftmargin=4ex]
\item 
$\StInd$ is a set, called the set of \textit{vertices} of $\Gr$.

\item 
$H = \mathop\bigsqcup\limits_{\stInd \in \StInd} \Bigl( d_{-1}(\stInd) \,\sqcup\, d_{+1}(\stInd) \Bigr)$ is a family of mutually disjoint at most countable linearly ordered sets, $d_{-1}(\stInd)$ and $d_{+1}(\stInd)$, called \textit{half-edges incident to $\stInd$}. 
We also denote $d(\stInd) = d_{-}(\stInd) \sqcup d_{+}(\stInd)$.

\item
$\xi: H \to H$ is an involution, i.e.\! a bijection such that $\xi^2 = \id_{H}$.
In this case if $\bdX\not=\xi(\bdX)$ for some $\bdX\in H$, then the \textit{unordered} pair $\edge{\bdX}{\xi(\bdX)}$ is called a \textit{closed edge} of $\Gr$.
Otherwise $\bdX$ is fixed point of $\xi$ and is called a \textit{half-open} edge of $\Gr$.

\item 
$\sigma: E \to \mZZ$ is a map from the set 
\[E = \bigl\{ \edge{\bdX}{\xi(\bdX)} \mid \bdX\in H, \ \bdX \not=\xi(\bdX) \bigr\}\]
of all closed edges of $\Gr$ to $\mZZ$, called \myemph{orientation of gluing}.

Equivalently, $\sigma$ can be regarded as a map $\sigma:H \setminus\mathrm{Fix}(\xi) \to \mZZ$ such that $\sigma\circ\xi=\sigma$.
\end{itemize}

\begin{definition}\label{def:striped_graphs_iso}
Let $\Gr =(\StInd, H, \xi, \sigma)$ and $\Gr'=(\StInd', H', \xi', \sigma')$
be graphs of striped atlases of some striped surfaces.
Then by an \myemph{isomorphism} of these graphs we will mean four maps
\begin{align*} 
&\viso: \StInd \to \StInd', &
&\eiso: H \to H', &
&\lori, \tori: \StInd \to \mZZ,
\end{align*}
having the following properties.
\begin{enumerate}[leftmargin=*, label=(\alph*)]
\item\label{enum:def:striped_graphs_iso:vertices}
$\viso$ and $\eiso$ are bijections satisfying the identity
\[
\eiso(d_{s}(\stInd)) = d'_{\tori(\stInd) \cdot s}(\viso(\stInd))
\]
for all $\stInd\in\StInd$ and $s\in\mZZ$, where $d'_{\pm1}(\stInd') \subset H'$ is the set of half edges of $\Gr'$ incident to $\stInd'\in\StInd$.
Moreover, both bijections 
\begin{align*} 
& \eiso|_{d_{-1}(\stInd)}: d_{-1}(\stInd) \to d'_{-\tori(\stInd)}(\viso(\stInd)), \\
& \eiso|_{d_{+1}(\stInd)}: d_{+1}(\stInd) \to d'_{\tori(\stInd)}(\viso(\stInd)),
\end{align*}
are increasing for $\lori(\stInd) = +1$ and decreasing for $\lori(\stInd) = -1$.

\item\label{enum:def:striped_graphs_iso:half_edges}
$\xi'\circ\eiso = \eiso \circ \xi$, in particular, $\eiso$ induces a bijection between closed edges of $\Gr$ and $\Gr'$.

\item\label{enum:def:striped_graphs_iso:gluing_orientation}
Let $\edge{\bdX}{\bdY}$ be a closed edge of $\Gr$ with $\bdX \in d(\stInd)$ and $\bdY=\xi(\bdX) \in d(\mu)$ for some $\stInd, \mu \in\StInd$.
Then
\begin{equation}\label{equ:compat_cond}
\lori(\stInd) \cdot \sigma(\bdX,\bdY) = \sigma'(\eiso(\bdX),\eiso(\bdY)) \cdot \lori(\mu).
\end{equation}
\end{enumerate}
\end{definition}

Notice that the set $\AutG$ of all automorphisms of a graph $\Gr$ is a group with respect to the following multiplication: if 
\[ a'=(\viso', \eiso', \lori', \tori'), \  a=(\viso, \eiso, \lori, \tori)  \ \in \ \AutG,\]
then their product $a'' = a' a=(\viso'', \eiso'', \lori'', \tori'')$ is defined as follows:
\begin{align}
\label{equ:AutG_prj_LH}
\viso'' &= \viso' \circ\viso,  &
\eiso'' &= \eiso' \circ \eiso, \\
\label{equ:AutG_prj_Z22}
\lori''(\stInd) &= \lori'(\viso(\stInd)) \cdot \lori(\stInd), &
\tori''(\stInd) &= \tori'(\viso(\stInd)) \cdot \tori(\stInd),
\end{align}
for all $\stInd\in\StInd$.

Let $\mathbf{1}:\StInd\to\mZZ$ be the constant function taking value $+1$.
Then $(\id_{\StInd}, \id_{H}, \mathbf{1}, \mathbf{1})$ is the unit of $\AutG$ and $(\viso, \eiso, \lori, \tori)^{-1} = (\viso^{-1}, \eiso^{-1}, \lori, \tori)$.

For a set $X$ denote by $\Sigma(X)$ the group of all bijections of $X$, that is the permutation group on $X$.
For a group $A$ let also $A^X$ be the group of all maps $X \to A$ with respect to the point-wise multiplication.
Then the group $\Sigma(X)$ naturally acts from the right on $A^X$ by the rule: the result of the action of a bijection $\viso:X \to X$ from $\Sigma(X)$ on a map $a:X \to A$ belonging to $A^X$ is the composition map
\[
a\circ \viso: X \xrightarrow{~~\viso~~} X \xrightarrow{~~a~~} A.
\]
The corresponding semidirect product $A^X \rtimes \Sigma(X)$ is called the \textit{wreath} product of $\Sigma(X)$ and $A$ over $X$ and denoted by $A \wr_{X} \Sigma(X)$.
Thus, by definition, $A \wr_{X} \Sigma(X)$ is a direct product of sets $A^X \times \Sigma(X)$ with respect to the following multiplication:
\[
(a',\viso') (a, \viso) = \bigl((a'\circ\viso) \cdot a, \ \viso'\circ\viso),
\]
where $\cdot$ denotes multiplication in $A^X$.
Notice that there is a natural surjective homomorphism $\eta:A \wr_{X} \Sigma(X) \to \Sigma(X)$, $\eta(a,\viso) = \viso$, whose kernel is $A^X \times \id_{X}$.
Moreover, we also have an inclusion $\mathbf{1} \times \Sigma(X) \subset A \wr_{X} \Sigma(X)$, where $\mathbf{1}:X \to A$ is the constant map into the unit of $A$.
In other words the following short exact sequence
\[
1 \to A^X \to A \wr_{X} \Sigma(X) \xrightarrow{~~\eta~~}  \Sigma(X) \to 1
\]
admits a \textit{section} $s:\Sigma(X) \to A \wr_{X} \Sigma(X)$, $s(\viso) = (\mathbf{1},\viso)$, i.e. a homomorphism such that $\eta\circ s = \id(\Sigma(X))$.

Rewriting~\eqref{equ:AutG_prj_LH} and~\eqref{equ:AutG_prj_Z22} in the form:
\[
(\viso', \eiso', \lori', \tori')\ (\viso, \eiso, \lori, \tori) = 
\bigr(\viso' \circ\viso, \
\eiso' \circ \eiso, \ 
(\lori'\circ\viso) \cdot \lori, \
(\tori'\circ\viso) \cdot \tori
\bigl)
\]
we see that $\AutG$ is a \textit{subgroup} of
\[
\Bigl( \mZZ^2 \wr_{\StInd} \Sigma(\StInd) \Bigr) \times \Sigma(H).
\]

\begin{theorem}\label{th:iso_atlas_iso_graph}
Each equivalence of striped atlases induces an isomorphism between their graphs.
Conversely, each isomorphism between their graphs is induced by some striped atlases equivalence.
\end{theorem}

Before proving Theorem~\ref{th:iso_atlas_iso_graph} let us first consider several illustrating examples.
To preserve the formalism we need to talk about maps from empty set.
As usual, we identify a map $f:A \to B$ between sets with its graph 
$\{(a,f(a)) \mid a\in A\} \subset A\times B$.
Therefore a map $\varnothing \to B$ from empty set is an empty subset of the empty set $\varnothing \times B$.

\begin{example}\label{exmp:open_strip}
Let $\strip = \bR\times(-1,1)$ and $\qmap=\id_{\strip}:\strip \to\strip$ be a striped atlas consisting of one strip, see Figure~\ref{fig:atlas_of_one_strip}(a).
Then $\StInd=\{*\}$ consists of a unique point, $H = \varnothing$, and so $\xi:H \to H$ and $\sigma:H\setminus\mathrm{Fix}(\xi)\to\mZZ$ are maps of empty set.

Let  $(\viso, \eiso, \lori, \tori) \in \AutG$.
Then $\viso = \id_{\StInd}$ and $\eiso = \id_{H}$ are uniquely determined, while $\lori,\tori:\{*\}\to\mZZ$ can be arbitrary maps.
It easily follows that $\AutG \cong \mZZ \times \mZZ$.
\end{example}

\begin{figure}[ht]
\begin{tabular}{ccc}
\includegraphics[height=0.3cm]{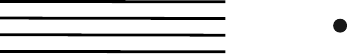} & \qquad\qquad\qquad & 
\includegraphics[height=0.7cm]{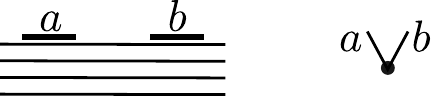}   \\
(a) & &  (b)
\end{tabular}
\caption{Striped atlases consisting of one strip and being the identity homeomorphisms}\label{fig:atlas_of_one_strip}
\end{figure}

\begin{example}\label{exmp:strip_with_2_bd}
Let $\strip = \bR\times(-1,1) \cup \{(-2,-1) \cup (1,2)\} \times \{1\}$ and again $\qmap=\id_{\strip}:\strip \to\strip$ be a striped atlas consisting of one strip, see Figure~\ref{fig:atlas_of_one_strip}(b).
Then $\StInd=\{*\}$ consists of a unique point, $H = \{a,b\} = d_{+1}(*)$, where $a=(-2,-1)\times\{1\}$, $b=(1,2)\times\{1\}$, and $a<b$ in the sense of the linear order in $d_{+1}(*)$.
As these intervals are not glued, we see that $\xi = \id_{H}:H \to H$ and so $\sigma:H\setminus\mathrm{Fix}(\xi)\to\mZZ$ is a map from empty set.

Let $x=(\viso, \eiso, \lori, \tori) \in \AutG$.
Then $\viso = \id_{\StInd}$.
Moreover, as $H = d_{+1}(*)$, and so $d_{-1}(*)=\varnothing$, it follows that $\eiso(d_{+1}(*)) = d_{+1}(*)$, whence $\tori(*) = +1$.

If $\eiso(a)=a$, then $\eiso=\id_{H}$, whence $\lori(*)=+1$, and so $x$ is the unit of $\AutG$.
Suppose $\eiso(a)=b$, then $\eiso(b)=a$, so $\eiso$ is an order reversing bijection of $H = d_{+1}(*)$, whence $\lori(*)=-1$.
Thus $\AutG$ consists of two elements, i.e.\! $\AutG \cong \mZZ$.
\end{example}

\begin{example}\label{exmp:open_cylinder}
Let $\strip=\bR\times[0,1]$, $\phi:\bR\times\{0\}\to\bR\times\{+1\}$ be a homeomorphism given by $\phi(x,0) = (x,1)$, then the quotient $\stripSurf = \strip / \phi$ is an \textit{open cylinder} $\bR\times S^1$, and the quotient map $\qmap:\strip \to\stripSurf$ is a striped atlas, see Figure~\ref{fig:cylindes_mobius_band}(a).

In this case $\StInd=\{*\}$ again consists of a unique point, $H = \{a,b\}$, where $a=\bR\times\{0\}$, $b=\bR\times\{1\}$, $\xi:H\to H$ is given $\xi(a)=b$, $\xi(b)=a$, and $\sigma:H\to\mZZ$ is defined by $\sigma(a)=\sigma(b)=\ori(\phi)=+1$.

Let $x=(\viso, \eiso, \lori, \tori) \in \AutG$.
Then $\viso = \id_{\StInd}$.
Moreover, since $\Gr$ has a unique edge $\edge{a}{b}$, $\eiso$ preserves this edge, whence it follows from~\eqref{equ:compat_cond} that $\lori(a) =\lori(b)$.

Suppose $\eiso(a)=a$, then $\eiso=\id_{H}$, whence $\tori(*)=+1$.
Otherwise, $\eiso(a)=b$, $\eiso(b)=a$, and $\tori(*)=-1$.
Notice that in both of those cases, the common value $\lori(a) =\lori(b)$ can be taken arbitrary.

This implies that $\AutG \cong \mZZ \times \mZZ$.
\end{example}

\begin{example}\label{exmp:open_Mobius_band}
Suppose as in the previous example $\strip=\bR\times[0,1]$, but now $\phi:\bR\times\{0\}\to\bR\times\{+1\}$ is given by $\phi(x,0) = (-x,1)$, and so it reverses orientation.
In this case the quotient $\stripSurf = \strip / \phi$ is an \textit{open M\"obius band}, see Figure~\ref{fig:cylindes_mobius_band}(b).
One easily check that $\AutG \cong \mZZ \times \mZZ$ as well.
\end{example}

\begin{figure}[ht]
\begin{tabular}{ccc}
\includegraphics[height=1.5cm]{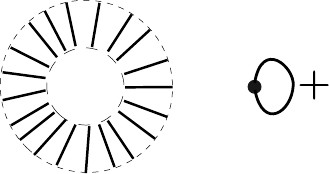} & \qquad\qquad\qquad & 
\includegraphics[height=1.5cm]{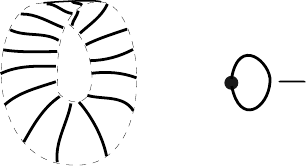}  \\
(a) & &  (b) 
\end{tabular}
\caption{Foliated open cylinder and M\"obius band}\label{fig:cylindes_mobius_band}
\end{figure}

\begin{proof}[Proof of Theorem~\ref{th:iso_atlas_iso_graph}]
Let 
\begin{align*}
&\qmap: \preStripSurf= \mathop{\sqcup}\limits_{\stInd\in\StInd}\strip_{\stInd} \to \stripSurf, &
&\qmap':\preStripSurf' = \mathop{\sqcup}\limits_{\stInd'\in\StInd'}\strip'_{\stInd'} \to \stripSurf'
\end{align*}
be striped atlases on surfaces $\stripSurf$ and $\stripSurf'$ respectively, and $\Gr=(\StInd, H, \xi, \sigma)$ and $\Gr'=(\StInd', H', \xi', \sigma')$ be the their graphs.

1) Suppose $(\dif,k)$ is a pair of homeomorphisms defining an equivalence of atlases, so we have a commutative diagram~\eqref{eq:strip_surf_equiv}.
Then $\dif$ induces a bijection between the connected components of $\preStripSurf$ and $\preStripSurf'$ which yields a bijection $\viso:\Lambda \to \Lambda'$ between the corresponding sets of indices (being in turn vertices of $\Gr$ and $\Gr'$) such that $\dif(\strip_{\stInd}) = \strip'_{\viso(\stInd)}$.

In particular, $\dif$ yields also a bijection between the boundary components of $\preStripSurf$ and $\preStripSurf'$ being sets of half edges of $\Gr$ and $\Gr'$.
Thus we get a bijection $\eiso:H \to H'$.

It remains to define the functions $\lori,\tori:\StInd\to\mZZ$.
Take $\stInd\in\StInd$ and consider the restriction 
$\dif|_{\strip_{\stInd}}:\strip_{\stInd}\to \strip'_{\viso(\stInd)}$.
Assume that $\Int\strip_{\stInd}=\bR\times(a,b)$ and $\Int\strip_{\viso(\stInd)}=\bR\times(c,d)$ for some $a<b, c<d\in\bR$.
Since $\dif|_{\strip_{\stInd}}$ preserves leaves being horizontal lines, we have that
\[\dif|_{\strip_{\stInd}}(x,y) = (\alpha(x,y), \beta(y))\]
where 
\begin{itemize}[leftmargin=4ex]
\item 
$\alpha:\strip_{\stInd} \to \bR$ is a continuous function such that for each $y\in(a,b)$ the correspondence $x \mapsto \alpha(x,y)$ is a homeomorphism $\alpha_y: \bR\to\bR$;
\item 
$\beta:(a,b) \to (c,d)$ is a homeomorphism.
\end{itemize}
Evidently, all homeomorphisms $\alpha_{y}$ are increasing or decreasing mutually for all $y\in(a,b)$, i.e. $\ori(\alpha_{y})$ does not depend on $y\in(a,b)$.
Therefore we set
\begin{align*}
\lori(\stInd) &= \ori(\alpha_y), &
\tori(\stInd) &= \ori(\beta).
\end{align*}

We claim that $(\viso, \eiso, \lori, \tori)$ is an isomorphism between graphs $\Gr$ and $\Gr'$ in the sense of Definition~\ref{def:striped_graphs_iso}.

Notice that the restriction $\dif|_{\partial\strip_{\stInd}}: \partial\strip_{\stInd} \to \partial\strip'_{\viso(\stInd)}$ is a monotone homeomorphism which easily implies conditions~\ref{enum:def:striped_graphs_iso:vertices} and~\ref{enum:def:striped_graphs_iso:half_edges} of Definition~\ref{def:striped_graphs_iso}.
We leave the verification for the reader and will check condition~\ref{enum:def:striped_graphs_iso:gluing_orientation} only.

Let $\edge{\bdX}{\bdY}$ be a closed edge of $\Gr$ with $\bdX \in d(\stInd)$ and $\bdY=\xi(\bdX) \in d(\mu)$ for some $\stInd, \mu \in\StInd$, and 
$\bdX'=\eiso(\bdX)$ and $\bdY'=\eiso(\bdY)$.
This means that $\bdX \subset \partial \strip_{\stInd}$ and $\bdY \subset \partial \strip_{\mu}$ are boundary components with $\qmap(\bdX) = \qmap(\bdY)$, $\bdX' = \dif(\bdX) \subset \partial \strip'_{\viso(\stInd)}$, and $\bdY' = \dif(\bdY) \subset \partial \strip'_{\viso(\mu)}$.
Then we have the following commutative diagram:
\[
\begin{CD}
\bdY @>{\dif|_{\bdY}}>{~~~~~~~~~~\ori(\dif|_{\bdY}) = \lori(\mu)~~~~~~~~~~}> \bdY' \\
@V{\phi}V{\ori(\phi)  = \sigma(\bdX,\bdY)}V @V{\ori(\phi')= \sigma'(\bdX',\bdY')}V{\phi'}V \\
\bdX @>~~~~~~~~~~\ori(\dif|_{\bdX}) = \lori(\stInd) ~~~~~~~~~~>{\dif|_{\bdX}}> \bdX'
\end{CD}
\]
where $\phi$ and $\phi'$ are gluing homeomorphisms.
Hence
\begin{align*}
\lori(\stInd) \cdot \sigma(\bdX,\bdY) &=
\ori(\dif|_{\bdY})\cdot \ori(\phi) =  \ori(\phi\circ \dif|_{\bdY}) = \\
&= \ori(\dif|_{\bdX} \circ \phi') =
\ori(\dif|_{\bdX}) \cdot \ori(\phi') =
\lori(\mu) \cdot \sigma'(\bdX',\bdY').
\end{align*}

2) To prove the converse statement, notice that due to Theorem~\ref{th:all_striped_surf_are_affine}, one can assume in addition that both atlases $\qmap$ and $\qmap'$ are affine.

Let $(\viso,\eiso,\lori, \tori)$ be an isomorphism between $\Gr$ and $\Gr'$ in the sense of Definition~\ref{def:striped_graphs_iso}.

Let $\stInd \in \StInd$ and $\stInd' = \viso(\stInd)$.
We will now construct a homeomorphism $\dif_{\stInd}:\strip_{\stInd} \to \strip'_{\stInd'}$ in the following way.

(i) First suppose $\partial\strip_{\stInd}=\varnothing$, that is $d(\stInd)=\varnothing$.
Since $\eiso$ bijectively maps $d(\stInd)$ onto $d'(\stInd')$, it follows that $d'(\stInd')=\varnothing$, and so $\partial\strip'_{\nu(\stInd)}=\varnothing$ as well.
Not loosing generality, one can assume that $\strip_{\stInd}=\strip'_{\stInd'}=\bR\times(-1,1)$.
Then we define $\dif_{\stInd}$ by the formula:
\[
\dif_{\stInd}(x,y) = \bigl( \lori(\stInd) x, \tori(\stInd) y ).
\]

(ii) Now assume that $\partial\strip_{\stInd}\not=\varnothing$.
Let $\bdX \in d(\stInd)$ be a half-edge in $\Gr$ incident to the vertex $\stInd$, that is $\bdX$ is a boundary component of $\strip_{\stInd}$.
Then $\viso(\bdX)$ is a boundary interval of $\strip'_{\viso(\stInd)}$.
Since we assumed that strips $\strip_{\stInd}$ and $\strip'_{\viso(\stInd)}$ are model, the intervals $\bdX$ and $\viso(\bdX)$ are bounded.
Define $\dif_{\stInd}$ on $\bdX$ to be a unique affine homeomorphism $\psi_{\bdX}:\bdX \to \viso(\bdX)$ with $\ori(\psi_{\bdX}) = \lori(\stInd)$.

The family of all $\{\psi_{\bdX}\}_{\bdX\in d(\stInd)}$ give a homeomorphism $\dif_{\stInd}:\partial\strip_{\stInd} \to \partial\strip'_{\viso(\stInd)}$.
Due to property~\ref{enum:def:striped_graphs_iso:vertices} of Definition~\ref{def:striped_graphs_iso}, $\dif_{\stInd}$ is monotone, and therefore by Theorem~\ref{th:homeomorphism_of_strips} $\dif$ extends to a foliated homeomorphism $\dif_{\stInd}:\strip_{\stInd} \to \strip'_{\viso(\stInd)}$.

Thus we obtain a foliated homeomorphism 
$\dif: \mathop{\sqcup}\limits_{\stInd\in\StInd}\strip_{\stInd} \to \mathop{\sqcup}\limits_{\stInd'\in\StInd'}\strip'_{\stInd'}$
defined by $\dif|_{\strip_{\stInd}} = \dif_{\stInd}$ for $\stInd\in\StInd$.

We claim that \textit{$\dif$ induces a foliated homeomorphism $k:\stripSurf\to \stripSurf'$ such that the pair $(\dif,k)$ is an equivalence of striped atlases $\qmap$ and $\qmap'$}.

Let $D \subset \preStripSurf$, (resp. $D' \subset \preStripSurf'$), be the set of boundary intervals on which $\qmap$, (resp. $\qmap'$), is not injective.
Then $\dif$ yields a homeomorphism of $\preStripSurf\setminus D$ onto $\preStripSurf'\setminus D'$, whence the restriction $k:\qmap(\preStripSurf\setminus D) \to \qmap'(\preStripSurf'\setminus D')$ must be given by $k = \qmap' \circ \dif\circ\qmap^{-1}$.

Therefore it remains to show that $\dif$ is ``compatible'' with $\qmap$ and $\qmap'$ on $D$ and $D'$ in the sense that for each pair of boundary intervals $\bdX\subset\partial\strip_{\stInd}$ and $\bdY\subset\partial\strip_{\mu}$ with $\qmap(\bdX) = \qmap(\bdY)$, 
we have that $\qmap'(\dif(\bdX))=\qmap'(\dif(\bdY))$ and the following commutative diagram holds true:
\begin{equation}\label{equ:constr_k}
\xymatrix{
\qmap(\bdX)\ar@{=}[d] && \bdX \ar[ll]_-{\qmap} \ar[rrr]^-{\dif|_{\bdX} = \psi_{\bdX}} \ar[d]_-{\phi} 
&&& \bdX' \ar[rr]^-{\qmap'} \ar[d]_{\phi'} && \qmap'(\bdX')\ar@{=}[d] \\
\qmap(\bdY) && \bdY \ar[ll]_-{\qmap} \ar[rrr]^-{\dif|_{\bdY} = \psi_{\bdY}}
&&& \bdY' \ar[rr]^-{\qmap'}        && \qmap'(\bdY')
}
\end{equation}
where $\phi$ and $\phi'$ are gluing homeomorphisms.
Then for each $z\in D$ we will set $k(\qmap(z)) = \qmap'\circ \dif(z)$.

In term of graphs we have that $\edge{\bdX}{\bdY}$ is a closed edge of $\Gr$ such that $\bdX\in d(\stInd)$, $\bdY=\xi(\bdX)\in d(\mu)$, $\bdX'=\eiso(\bdX)$, and $\bdY'=\eiso(\bdY)$.
Then by \ref{enum:def:striped_graphs_iso:half_edges}
\[
\bdY' = \eiso(\bdY) = \eiso\circ \xi(\bdX) = \xi' \circ \eiso(\bdX) = \xi'(\bdX)
\]
and so $\{\bdX', \bdY'\}$ is a closed edge of $\Gr'$, that is $\qmap'(\dif(\bdX))=\qmap'(\dif(\bdY))$.

Then we have the diagram~\eqref{equ:constr_k} but need to check commutativity of its central square consisting of affine homeomorphisms.
It follows from~\ref{enum:def:striped_graphs_iso:gluing_orientation} that $\ori(\phi' \circ \psi_{\bdX}) = \ori(\psi_{\bdY}\circ \phi)$.
Since $\phi' \circ \psi_{\bdX}, \,\psi_{\bdY}\circ \phi: \bdX \to \bdY'$ are \textit{affine} homeomorphisms, it follows that they coincide, and so diagram~\eqref{equ:constr_k} is commutative.

Thus $(\dif,k)$ is an equivalence of striped atlases inducing given isomorphism $(\viso,\eiso,\lori, \tori)$ between $\Gr$ and $\Gr'$. 
\end{proof}

\section{Characterization of a certain class of striped surfaces}\label{sect:charact_str_prop}
Let $(\stripSurf,\Partition)$ be a foliated surface with countable base, $\stripSurf/\Partition$ the set of leaves of $\Partition$, and $p:\stripSurf\to\stripSurf/\Partition$ be the quotient map.
We will endow $\stripSurf/\Partition$ with the quotient topology, so a subset $\Vsp\subset\stripSurf/\Partition$ is open if and only if $p^{-1}(\Vsp)$ is open in $\stripSurf$.
Notice that a priori $\stripSurf/\Partition$ is not even a $T_0$-space.

For each leaf $\leaf$ of $\Partition$ let $\JL=[0,1)$ if $\omega\subset\partial\stripSurf$ and $\JL=(-1,1)$ otherwise.
Then a \myemph{cross-section} of $\Partition$ passing through $\omega$ is a continuous map $\gamma:\JL\to\stripSurf$ such that $\gamma(0)\in\leaf$ and for distinct $s,t\in \JL$ their images $\gamma(s)$ and $\gamma(t)$ belong to distinct leaves of $\Partition$.

A subset $\Usp\subset\stripSurf$ is called \myemph{saturated} if it is a union of leaves.
For each leaf $\omega\in\Partition$ denote by $\hcl{\omega}$ the intersection of closures of all saturated neighbourhoods of $\omega$.
Evidently $\omega\subset \hcl{\omega}$.

\begin{definition}\label{def:spec_leaf}\cite{MaksymenkoPolulyakh:PGC:2015}
A leaf $\omega$ will be called \myemph{special} whenever $\omega \not= \hcl{\omega}$, see Figure~\ref{fig:spec_leaves}.
We will denote by $\Sigma$ the family of all special leaves of $\Partition$.
\end{definition}
In~\cite{HaefligerReeb:EM:1957} and~\cite{GodbillonReeb:EM:1966} special leaves were called \myemph{branch} points of $\stripSurf/\Partition$.

\begin{figure}[h]
\center{\includegraphics[height=1.8cm]{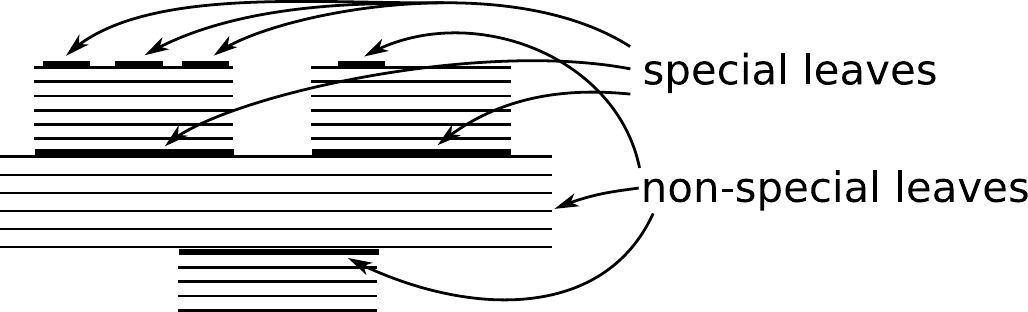}}
\caption{Special and non-special leaves}
\label{fig:spec_leaves}
\end{figure}

\begin{lemma}\label{lm:D_Sigma_prop}
Suppose there exist a striped atlas $\qmap:\mathop{\sqcup}\limits_{\stInd\in\StInd} \strip_{\stInd} \to \stripSurf$ such that $\Partition$ is its canonical foliation.
Let also $D = \qmap\Bigl( \mathop{\sqcup}\limits_{\stInd\in\StInd} \partial\strip_{\stInd}\Bigr)$ be the union of images of boundary components of strips.
Then
\begin{enumerate}[leftmargin=*, label={\rm(\roman*)}]
\item\label{enum:lm:D_Sigma_prop:SdZ_D}
 $\Sigma\cup\partial\stripSurf \ \subset \ D$;
\item\label{enum:lm:D_Sigma_prop:S_dZ_D_loc_fin}
$\Sigma$, $\partial\stripSurf$, and $D$ are a locally finite families of leaves;
\item\label{enum:lm:D_Sigma_prop:S_dZ_D_closed} $\Sigma$, $\partial\stripSurf$, and $D$ are closed subsets of $\stripSurf$.
\end{enumerate}
\end{lemma}
\begin{proof}
\ref{enum:lm:D_Sigma_prop:SdZ_D}
By definition $\partial\stripSurf\subset D$.
Moreover, one easily check that a leaf $\omega\in\Partition$ is special, i.e.\! $\omega\subset\Sigma$, if and only if there exists a boundary interval $\bdX\subset \partial_{\epsilon}\strip_{\stInd}$ for some $\stInd\in\StInd$ and $\epsilon\in\{\pm\}$ such that $\qmap(\bdX) = \omega$ and $\bdX \not=\partial_{\epsilon}\strip_{\stInd}$.
Hence $\Sigma\subset D$ as well.

\ref{enum:lm:D_Sigma_prop:S_dZ_D_loc_fin}.
Evidently, each leaf $\omega$ in $D$ has an open neighbourhood containing no other leaves from $D$.
This implies that $D$ is a locally finite family of closed subsets of $\stripSurf$, whence so any subfamily of $D$.
In particular, this holds for $\Sigma$ and $\partial\stripSurf$.

\ref{enum:lm:D_Sigma_prop:S_dZ_D_closed} follows from~\ref{enum:lm:D_Sigma_prop:S_dZ_D_loc_fin}, since each leaf of $\Partition$ is a closed subset of $\stripSurf$.
\end{proof}

A striped atlas on $\stripSurf$ will be called \myemph{reduced} whenever $D = \Sigma \cup \partial\stripSurf$.

\begin{theorem}\label{th:atlas_to_reduced_form}
{\rm\cite[Theorem~3.7]{MaksymenkoPolulyakh:PGC:2015}.}
Let $\stripSurf$ be a striped surface with countable base.
Then one of the following statements holds true: either 
\begin{enumerate}[label={\rm(\arabic*)}]
\item $\stripSurf$ is foliated homeomorphic to the open cylinder or M\"obius band from Examples~\ref{exmp:open_cylinder} and~\ref{exmp:open_Mobius_band}, or
\item $\stripSurf$ admits a reduced atlas.
\end{enumerate}  
\end{theorem}
\begin{proof}[Idea of proof.]
We briefly discuss the proof in terms of the graph $\Gr$ of the striped atlas $\qmap$.
It will be convenient to say that an edge $\{\bdX,\bdY\}$ of $\Gr$ is \myemph{unessential} whenever $\bdX = \partial_{\epsilon} \strip$ and $\bdY = \partial_{\epsilon'} \strip'$ for some \myemph{distinct} strips $\strip,\strip'$ of the atlas and some $\epsilon,\epsilon' \in \{\pm\}$, see Figure~\ref{fig:unessential_edge}.

\begin{figure}[h]
\includegraphics[height=3.5cm]{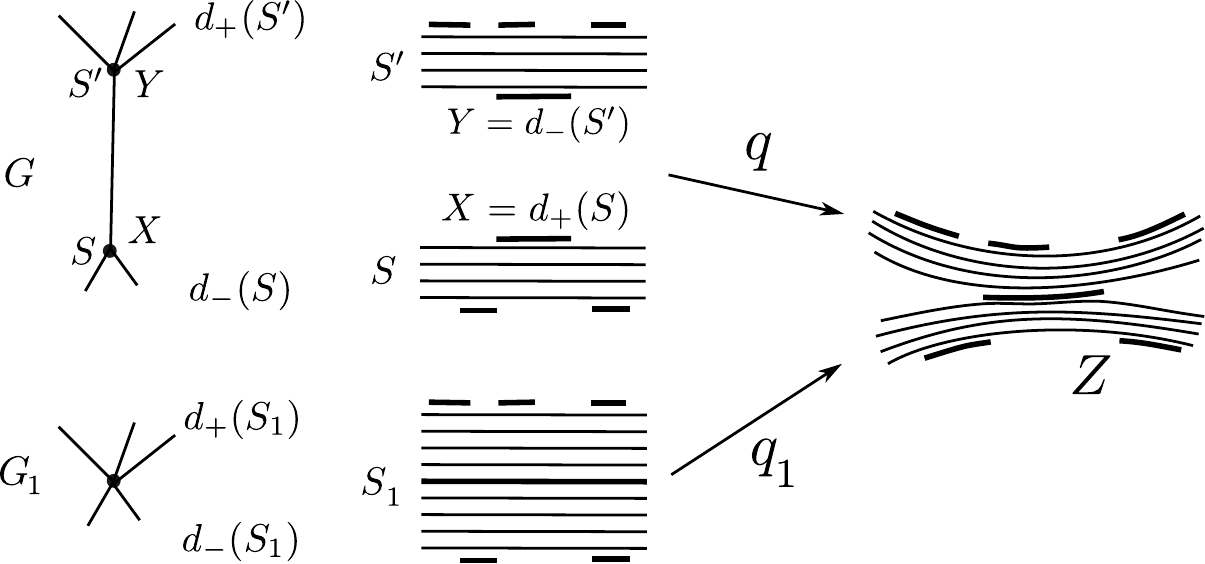}
\caption{Reduction of non-special leaves from $D$}\label{fig:unessential_edge}
\end{figure}

In particular, each unessential edge corresponds to a non-special leaf $\omega \subset D\setminus(\Sigma\cup\partial\stripSurf)$.
The principal observation of Theorem~\ref{th:atlas_to_reduced_form} is that gluing $\strip$ and $\strip'$ along $X$ and $Y$ gives again a strip $\strip_1$, see \cite[Lemma~3.2]{MaksymenkoPolulyakh:PGC:2015}.
Therefore one can replace $\strip$ and $\strip'$ in the atlas $\qmap$ with $\strip_1$.

On the graph this means that we replace a closed edge between $\strip$ and $\strip'$ with one vertex.
That techniques also allows to eliminate even countable paths of such edges.
Hence if $\Gr$ does not contain finite cycles of unessential edges, then, using the assumption that $\stripSurf$ has a countable base, one can remove all unessential edges and obtain a reduced atlas.

However, if there is a finite cycle of unessential edges, then one can remove all of them but one.
This gives two special surfaces: open cylinder and M\"obius band from Examples~\ref{exmp:open_cylinder} and~\ref{exmp:open_Mobius_band} in which we glue $\bdX=\partial_{-}\strip$ with $\bdY=\partial_{+}\strip$.
But the corresponding closed edge $\edge{\bdX}{\bdY}$ is not unessential, since now $\bdX$ and $\bdY$ belong to \textit{the same} strip.
\end{proof}

\newcommand\condLF{\mathrm{(\Sigma LocFin)}}
\newcommand\condU{\mathrm{(SatNbh)}}
\newcommand\condCR{\mathrm{(CrossSect)}}
\newcommand\condST{\mathrm{(StrAtlas)}}
\newcommand\condLT{\mathrm{(PrjLocTriv)}}

Consider the following five conditions on $(\stripSurf,\Partition)$.

\begin{enumerate}[leftmargin=8ex, itemsep=1ex, align=left]
\item[$\condST$:]
$\stripSurf$ admits a striped atlas whose canonical foliation is $\Partition$.

\item[$\condLF$:]
The family $\Sigma$ of all special leaves of $\Partition$ is locally finite.

\item[$\condLT$:]
The quotient map $p:\stripSurf\to\stripSurf/\Partition$ is a locally trivial fibration and the space of leaves $\stripSurf/\Partition$ is locally homeomorphic with $[0,1)$ (though it is not in general a Hausdorff space).

\item[$\condU$:]
For each leaf $\leaf\in\Partition$ there exist an open $\Partition$-saturated neighbourhood $\Usp$ of $\omega$ and a homeomorphism $\eta: \bR\times \JL \to \Usp$ such that $\eta(\bR\times t)$ is a leaf of $\Partition$ and $\eta(\bR\times 0) = \omega$.

\item[$\condCR$:]
Each leaf $\leaf\in\Partition$ has a cross-section passing through $\leaf$.
\end{enumerate}

The following statement summarizes relations between the above properties obtained in~\cite{MaksymenkoPolulyakh:PGC:2015}, \cite{MaksymenkoPolulyakh:MFAT:2016}, \cite{MaksymenkoPolulyakh:PGC:2016}, and in the present paper.

\begin{theorem}\label{th:rel_between_conditions}{\rm\cite{MaksymenkoPolulyakh:MFAT:2016}, \cite{MaksymenkoPolulyakh:PGC:2016}}.
Let $(\stripSurf,\Partition)$ be a foliated surface satisfying the following two conditions:
\begin{enumerate}[label={\rm(\roman*)}]
\item\label{enum:genfolsurf:closed_leaves}
each leaf of $\Partition$ is a non-compact closed subset of $\stripSurf$;
\item\label{enum:genfolsurf:bd_leaves}
each boundary component of $\stripSurf$ is a leaf of $\Partition$.
\end{enumerate}
Then we have the following implications:
\begin{itemize}[leftmargin=*]
\item $\condLT \Rightarrow \condU + \condCR$;
\item $\condST\Rightarrow\condLF\Rightarrow \bigl[ \condLT \Leftrightarrow \condU \Leftrightarrow \condCR \bigr]$;
\item $\condLF + \condU \Rightarrow \condST$.
\end{itemize}
In particular, if either $\condU$ or $\condCR$ hold, then the conditions $\condLF$ and $\condST$ are equivalent.
\end{theorem}
\begin{proof}
The implication $\condLT \Rightarrow \condU + \condCR$ and the equivalence of conditions $\condU$, $\condLT$, and $\condCR$ under assumption $\condLF$ is proved in~\cite[Theorem~2.8]{MaksymenkoPolulyakh:MFAT:2016}.

The implication $\condST\Rightarrow\condLF$ is contained in statement~\ref{enum:lm:D_Sigma_prop:S_dZ_D_loc_fin} of Lemma~\ref{lm:D_Sigma_prop}.

Finally, the implication $\condLF + \condU \Rightarrow \condST$ is established in~\cite[Theorem~1.8]{MaksymenkoPolulyakh:PGC:2016}.
\end{proof}

\begin{remark}
For a striped surface $(\stripSurf,\Partition)$ with a striped atlas $\qmap$ condition~$\condU$ is equivalent to the requirement that \textit{$\qmap$ does not glue together boundary intervals belonging to the same side of the same strip}.
More precisely, if $\qmap(\Xsp) = \qmap(\Ysp)$ for some distinct boundary intervals $\Xsp \subset \partial_{\epsilon}\strip$ and $\Ysp \subset \partial_{\epsilon'}\strip'$, then either $\strip\not=\strip'$ or $\epsilon\not=\epsilon'$.
\end{remark}

\subsection{Foliated surface that does not admit a striped atlas}
Consider the sequence $z_n = (0,\tfrac{1}{n})$, $n\in\bN$, of points of $y$-axis on the plane converging to the origin $O$ and put
$K =\{z_n\}_{n\in\bN} \cup O$.
Let also $\stripSurf = \bR^2\setminus K$.
Then $\stripSurf$ admits a foliation $\Partition$ into non-compact leaves being connected components of the intersection of $\stripSurf$ with horizontal lines.
\begin{figure}[h]
\includegraphics[height=2.5cm]{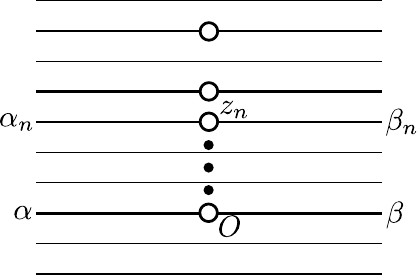}
\caption{}\label{fig:example_nostriped_atlas}
\end{figure}
\begin{sublemma}\label{lm:example_of_non-striped_surface}
The pair $(\stripSurf,\Partition)$ satisfies condition $\condCR$ and violates $\condLF$.
Hence it also violates $\condST$, that is $\stripSurf$ does not admit a striped atlas for which $\Partition$ is a canonical foliation.
\end{sublemma}
\begin{proof}
$\condCR$.
For each leaf $\leaf\in\Partition$ there exists a cross-section being just an one vertical interval in $\stripSurf$ transversal to $\leaf$.

To show that $\condLF$ fails, denote 
\begin{align*}
\alpha_{n} &= (-\infty,0) \times z_n, &
\beta_{n} &= (0, +\infty) \times z_n, \\
\alpha &= (-\infty,0) \times O, &
\beta &= (0, +\infty) \times O.
\end{align*}
Then $\Sigma = \{ \alpha_{n}, \beta_{n}\}_{n\in\bN} \cup \{ \alpha, \beta\}$ is the family of all special leaves of $\Partition$.
Evidently,
\begin{align*}
\hcl{\alpha_n}&=\hcl{\beta_n}=\{ \alpha_n, \beta_n\}, \ n\in\bN, &
\hcl{\alpha}&=\hcl{\beta}=\{ \alpha, \beta\},
\end{align*}
whence $\Sigma$ is not locally finite, since $\alpha_n$ converges to $\alpha$ and $\beta_n$ converges to $\beta$.
Therefore by Theorem~\ref{th:rel_between_conditions} $\stripSurf$ does not admit a striped atlas with a canonical foliation $\Partition$.
\end{proof}

Notice also that $\stripSurf'=\stripSurf\setminus\{\alpha,\beta\}$ is disconnected and each of its connected components admits a striped atlas.

\subsection{Foliation on the plane that does not admit a striped atlas}
We will construct a more complicated example on the plane $\bR^2$.
Consider the foliation $\Partition_0$ on the strip $\strip = \bR\times[0,1]$ shown in Figure~\ref{fig:example_on_r2}(a).
As indicated in Figure~\ref{fig:example_on_r2}(b) it is glued from four strips.
\begin{figure}[h]
\begin{tabular}{ccc}
\includegraphics[height=1.8cm]{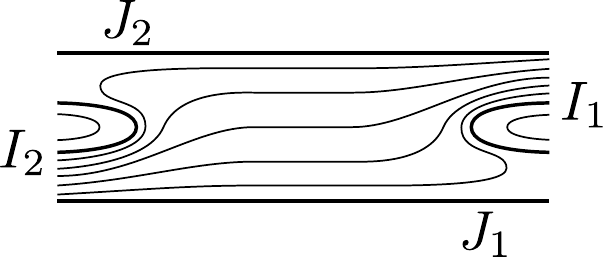} & \qquad\qquad &
\includegraphics[height=1.8cm]{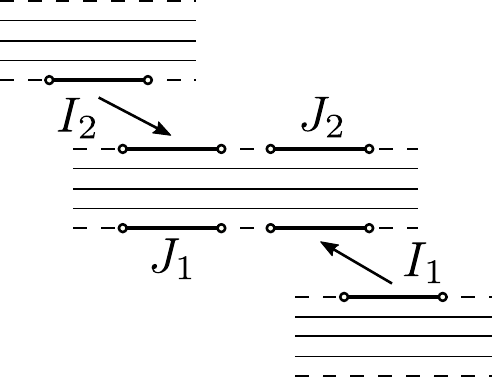} \\
(a) & & (b)
\end{tabular}
\caption{}\label{fig:example_on_r2}
\end{figure}
For each $k\in\bN$ define the following strip $\strip_k=\bR\times[\tfrac{1}{2^{k}}, \tfrac{1}{2^{k-1}}]$ and a homeomorphism 
\begin{align*}
&\phi_n:\strip \to \strip_k, &
\phi_k(x,y) &= (k x, (y+1)/ 2^{k} ),
\end{align*}
so it expands strip along $x$-axis and shrinks it along $y$-axis.

Denote by $\Partition_k$ foliation on $\strip_k$ being the image of $\Partition_0$ under $\phi_k$.
Then the union of all $\Partition_k$ gives a foliation on $\bR\times[0,1]$ which extends to the foliation on all of $\bR^2$ by parallel lines $\bR\times y$ for $y\in(-\infty,0]\cup(1,+\infty)$.
We will denote that foliation on $\bR^2$ by $\Partition$.

Then $\Sigma = \{ \phi_k(I_1), \phi_k(I_2), \phi_k(J_1) = \phi_{k+1}(J_2) \}_{k\in\bN}$ is the family of all special leaves of $\Partition$.
This set is not locally finite since the leaves $\phi_k(J_1)$ converge to the leaf $\bR\times 0$.
One easily check that $\Partition$ satisfies $\condCR$, whence by Theorem~\ref{th:rel_between_conditions} $(\bR^2,\Partition)$ does not admit a striped atlas.

\section{Homeotopy group of a canonical foliation}\label{sect:homeotopy_group}
Let $\stripSurf$ be a connected striped surface with a canonical foliation $\Partition$.
Denote by $\Homeo(\Partition)$ the group of all foliated homeomorphisms of $(\stripSurf,\Partition)$.
Thus, by definition, $\Homeo(\Partition)$ consists of all homeomorphism $\dif:\stripSurf\to\stripSurf$ such that for each leaf $\omega\in\Partition$ its image $\dif(\omega)$ is a leaf of $\Partition$ as well.
Endow $\Homeo(\Partition)$ with the compact open topology and let $\Homeo_0(\Partition)$ be the identity path component of $\Homeo(\Partition)$.
It consists of all homeomorphisms $\dif\in\Homeo(\Partition)$ isotopic to $\id_{\stripSurf}$ in $\Homeo(\Partition)$.
Then $\Homeo_0(\Partition)$ is a normal subgroup of $\Homeo(\Partition)$ and the quotient $\Homeo(\Partition) /\Homeo_0(\Partition)$ can be identified with the set $\pi_0\Homeo(\Partition)$ of all path components of $\Homeo(\Partition)$, that is $\pi_0\Homeo(\Partition) = \Homeo(\Partition) / \Homeo_0(\Partition)$.
This group will be called the \myemph{homeotopy} group of the foliation $\Partition$.

\begin{theorem}\label{th:HZ_AutG}{\rm c.f.\,\cite[Theorem~4.4]{MaksymenkoPolulyakh:PGC:2015}.}
Let $\qmap: \mathop{\sqcup}\limits_{\stInd\in\StInd}\strip_{\stInd} \to \stripSurf$ be a \myemph{reduced affine} striped atlas on a \myemph{connected} surface $\stripSurf$, $\Gr$ be its graph, and $\Partition$ be the corresponding canonical foliation.
\begin{enumerate}[leftmargin=2em, label={\rm(\roman*)}]
\item\label{enum:th:HZ_AutG:CylMob_HS_S1} 
If $\stripSurf$ is foliated homeomorphic either to an open cylinder from Example~\ref{exmp:open_cylinder} or a M\"obius band from Example~\ref{exmp:open_Mobius_band} then $\HZS$ is homotopy equivalent to the circle $S^1$.
\item\label{enum:th:HZ_AutG:Other_Cases_HS_pt}
Otherwise, $\HZS$ is contractible.
\end{enumerate}
In all the cases we have an isomorphism $\rho:\pi_0\HS \cong \AutG$.
\end{theorem}
\begin{proof}
\ref{enum:th:HZ_AutG:CylMob_HS_S1}
Suppose $\stripSurf$ is either an open cylinder or a M\"obius band.
Since by Examples~\ref{exmp:open_cylinder} and~\ref{exmp:open_Mobius_band} $\AutG \cong \mZZ \times \mZZ$ we need only to show that 
$\HZS$ is homotopy equivalent to the circle and $\pi_0\HS \cong \mZZ \times \mZZ$ as well.
We leave this statement as an exercise for the reader.

\ref{enum:th:HZ_AutG:Other_Cases_HS_pt}.
Now let $\stripSurf$ be \textit{neither} an open foliated cylinder \textit{nor} a foliated M\"obius band.
Then the following statement is a reformulation of~\cite[Theorem~4.4]{MaksymenkoPolulyakh:PGC:2015} in terms of striped atlases and their graphs.
In particular, it contains~\ref{enum:th:HZ_AutG:Other_Cases_HS_pt}. 
\begin{sublemma}\label{lm:hom_type_of_HZS}
{\rm c.f. \cite[Theorem~4.4]{MaksymenkoPolulyakh:PGC:2015}.}
For each $\kdif\in\HS$ there exists a unique homeomorphism $\dif:\preStripSurf \to \preStripSurf$ such that $\qmap \circ \dif = \kdif\circ\qmap$, i.e.\! $(\dif,\kdif)$ is a self-equivalence of the atlas $\qmap$.
Moreover, $\kdif\in\HZS$ if and only if $(\dif,\kdif)$ induces the identity automorphism of $\Gr$.

Also the group $\HZS$ is contractible.\qed
\end{sublemma}

It remains to construct an isomorphism $\rho:\pi_0\HS \cong \AutG$.
Let $\kdif\in\HS$ and $(\dif,\kdif)$ be the self-equivalence of the atlas $\qmap$.
Denote by $\rho(\kdif)$ the automorphism of $\Gr$ induced by $(\dif,\kdif)$, see Theorem~\ref{th:iso_atlas_iso_graph}.
Then one easily check that the correspondence $\kdif\mapsto\rho(\kdif)$ is a homomorphism $\rho:\pi_0\HS \cong \AutG$.

Moreover, by Theorem~\ref{th:iso_atlas_iso_graph} $\rho$ is surjective, and by Lemma~\ref{lm:hom_type_of_HZS} its kernel is $\HZS$.
This gives the required isomorphism 
\[ \pi_0\HS =  \HS/\HZS =  \HS/\ker(\rho) \cong \AutG.\]
Theorem~\ref{th:HZ_AutG} is completed.
\end{proof}

\bigskip

\begin{flushleft}
{\sc Sergiy Maksymenko} \\
Institute of Mathematics of NAS of Ukraine, Tereshchenkivska str. 3, Kyiv, 01004, Ukraine, \\
email: \texttt{maks@imath.kiev.ua}   

\bigskip

{\sc Eugene Polulyakh} \\
Institute of Mathematics of NAS of Ukraine, Tereshchenkivska str. 3, Kyiv, 01004, Ukraine, \\
email: \texttt{polulyah@imath.kiev.ua} 

\bigskip

{\sc Yuliya Soroka} \\
Taras Shevchenko National University of Kyiv \\
email: \texttt{ladyice09@gmail.com}
\end{flushleft}

\end{document}